\documentclass{amsart}

\title{Sampling and Frequency Warping}

\author{Stefan Lafon} \address{Stefan Lafon, Google}
\email{stephane.lafon@gmail.com}

\author{Jacques L\'evy V\'ehel} \address{Jacques L\'evy V\'ehel, Inria
Rocquencourt\\Projet Fractales\\France}
\email{jacques.levy-vehel@inria.fr}

\author{Jacques Peyri\`ere} \address{Jacques Peyri\`ere, UMR 8628,
  CNRS, Universit\'e Paris-Sud, Université Paris-Saclay.  Universit\'e
  Paris-Sud, Math\'ematique b\^at.~425, 91405 Orsay Cedex, France.}
\email{jacques.peyriere@math.u-psud.fr}


\usepackage{graphicx} \usepackage{amsmath,amsthm} \usepackage{amssymb}
\usepackage[T1]{fontenc} \usepackage[ansinew]{inputenc}
\usepackage{amsfonts}

\newtheorem{theorem}{Theorem}
\newtheorem{lemma}[theorem]{Lemma}
\newtheorem{proposition}[theorem]{Proposition}

\theoremstyle{definition} \newtheorem{definition}[theorem]{Definition}

\newtheorem{corollary}[theorem]{Corollary}

\theoremstyle{remark} \newtheorem{remark}[theorem]{Remark}

\numberwithin{equation}{section}

\newcommand{\e}[1]{{\mathrm e}^{#1}}
\newcommand{\lo}[1]{\mbox{{\rm\small o}}\left(#1\right)}
\newcommand{\bo}[1]{{\mathcal O}\left(#1\right)}
\newcommand{\mi}{\mathrm{i}}
\newcommand{\dif}{\mathrm{d}}

\DeclareMathOperator{\sinc}{sinc}

\begin{document}

\begin{abstract}
Optimal sampling of non band-limited functions is an issue of great
importance that has attracted considerable attention. We propose to
tackle this problem through the use of a frequency warping: First, by
a non-linear shrinking of frequencies, the function is transformed
into a band-limited one.  One may then perform a decomposition in
Fourier series. This process gives rise to new orthonormal bases of
the Sobolev spaces ${\mathcal H}^\alpha$. When $\alpha$ is an integer,
these orthonormal bases can be expressed in terms of Laguerre
functions.  We study the reconstruction and speed of convergence
properties of the warping-based sampling. Besides theoretical
considerations, numerical experiments are performed.
\end{abstract}

\maketitle

\section{Introduction}\label{intro}

Sampling theory is a rich research topic that lies at the boundary
between harmonic analysis and signal processing. The basic result is
the Whittaker-Kotelnikov-Shannon theorem (WKS)
\cite{Kotelnikov,Shannon}: From a signal processing point of view,
it tells that it is possible to reconstruct exactly a band-limited
signal, i.e., a signal whose Fourier transform is compactly supported
in $(-1/2,1/2)$, from regularly spaced samples provided the sampling
frequency is not smaller than 1. From a theoretical point of view, it
asserts the identity between the space of functions in $L^2(\mathbb
R)$ whose Fourier transform vanishes outside $(-1/2,1/2)$ and the
space $\{ \sum_{k \in \mathbb Z} c_k \sinc(t-k), (c_k)_k \in l^2 \}$,
where $\sinc t = \sin \pi t/\pi t$ .

A huge literature has been devoted to generalizing the WKS theorem
in various directions. Classical extensions include the case of
non-regular sampling, multichannel sampling, or the sampling and
reconstruction of functions in more general spaces.
See~\cite{Feicht,Jerri,Marks,Marks2} for excellent reviews.

We propose the following treatment. We are given an increasing
function~$\varphi$ from the interval $(-\frac{\pi}{2},\frac{\pi}{2})$
onto ${\mathbb R}$. Then one can expand $\widehat{X}\circ \varphi$ in
Fourier series. The coefficients $a_n$ of this expansion bear the
whole information contained in $X$. Indeed, the original signal, when
it lies in a Sobolev space, can be decomposed as the sum of a series
$\displaystyle \sum_{n\in{\mathbb Z}} a_n\gamma_n$, where the
$\gamma_n$ form an orthogonal basis of the Sobolev space under
consideration.

First, in Section~2, we develop a general setting and define the
\emph{warping operators}. Afterwards, In Section~3, we particularise
the situation and are able to derive formulas for the aforementioned
$\gamma_n$ which involve known special functions.

In Section~4, we examine the speed of approximation and compare the
warping method to the usual sampling. In the last section, we give an
account of computations an simulations. In particular we test and
compare on several signals the efficiency of its representations
either by sampling or by warping.

\section{The general warping operators}\label{setting}

\subsection{Preliminaries and notation}

For the Fourier transform of a function $X\in L^1({\mathbb R})$, we
use the following convention: $${\widehat X}(\omega)=\int_{\mathbb R}
X(t)\e{-\mi t\omega}\dif t.$$

If $X$ and $Y$ are two functions in $L^2({\mathbb R})$, their scalar
product $\int_{\mathbb R} X(t)\,\overline{Y(t)}\,\dif t$ is denoted
by~$\langle X,Y\rangle$. We use the same notation if $X\in{\mathcal
S}'({\mathbb R})$ and $Y\in{\mathcal S}({\mathcal R})$.

The Heaviside function will be denoted by~$H$. The symbol~$\delta$
stands either for the Dirac measure or the Kronecker sumbol, depending
on the context.

If $u\in{\mathbb R}$ and $j\in{\mathbb N}$, we define the generalized
binomial coefficient to be
\begin{equation*}
\binom{u}{j}=\frac {u(u-1)...(u-j+1)}{j!}
\end{equation*}
if~$j>0$, and~1 if $j=0$. With this convention, for all $u\in{\mathbb
R}$ and $|x|<1$, one has $(1+x)^u = \sum_{j\ge 0} \binom{u}{j}\,x^j$.

We also have the following formula
\begin{equation}\label{binom}
\binom{-u}{j} = (-1)^j\,\binom{u+j-1}{j}.
\end{equation}

We will use some special functions, namely the Whittaker function~$W$
(\cite{magnus}, p.~88) and the generalized Laguerre polynomials
(\cite{szego}, p.~100)
\begin{equation*}
L_n^{(\alpha)}(x) = \frac{1}{n!}\, x^{-\alpha}\e{x}
\left(\frac{\dif}{\dif x}\right)^{n} \Bigl(x^{n+\alpha}\e{-x}\Bigr) =
\sum_{j=0}^{n} \binom{n+\alpha}{n-j}\,\frac{(-x)^{j}}{j!},
\end{equation*}
where~$n\in{\mathbb N}$ and $\alpha\in{\mathbb R}$.

When the parameter~$\alpha$ is larger than~$-1$, one has the following
orthogonality relation
\begin{equation*}
\int_{o}^{+\infty}
L_n^{(\alpha)}(x)\,L_m^{(\alpha)}(x)\,x^\alpha\e{-x}\, dx =
\Gamma(\alpha+1)\,\binom{n+\alpha}{n}\,\delta_{n,m}.
\end{equation*}

\subsection{The warping operators}

We are given~$\psi$, an increasing $C^1$ function from~${\mathbb R}$
onto an open interval~$I$ of~${\mathbb R}$, and~$\chi$ a positive
function defined on~$I$.

The length of the interval~$I$, when it is bounded, will be denoted
by~$|I|$ and the mapping reciprocal to~$\psi$ by~$\varphi$. We set
$\varpi = (\chi\circ\psi)^2.\psi'$.

Let ${\mathcal H}_{\varpi}$ stand for the set of distributions~$X$
on~${\mathbb R}$ the Fourier transform of which is a
function~$\widehat{X}$ such that\quad $\displaystyle \int_{\mathbb R}
|\widehat{X}(\omega)|^2\varpi(\omega)\,\dif\omega$ is finite. This
definition makes sense when $1/\varpi$ has a polynomial growth.

\textbf{\em We always assume that both $\boldsymbol{\varpi}$ and
$\mathbf 1\boldsymbol{/\varpi}$ have polynomial growth.}

If $X\in{\mathcal H}_{\varpi}$, one has
\begin{eqnarray}\label{isometry}
\int_{I} |\widehat{X}\bigl(\varphi(u)\bigr)\,\chi(u)|^2 \dif u &=&
\int_{\mathbb R} |\widehat{X}(\omega)\,
\chi\bigl(\psi(\omega)\bigr)|^2\psi'(\omega)\,\dif\omega\\ &=&
\int_{\mathbb R} |\widehat{X}(\omega)|^2 \varpi(\omega)\,\dif\omega <
+\infty.\nonumber
\end{eqnarray}

This legitimates the following definition.

\begin{definition}[Warping operator]\label{wop} Given $\psi$
and~$\chi$, we define the warping operator ${\mathcal T}^{\psi,\chi}$
as follows:
$${\mathcal T}^{\psi,\chi}: \left(
\begin{array}{ccc}
{\mathcal H}_{\varpi}&\longrightarrow &L^2(I,\dif u)\\
X&\longmapsto&(\widehat{X}\circ\varphi).\chi
\end{array}
\right)$$
\end{definition}

In view of Formula~\ref{isometry}, the operator~${\mathcal
T}^{\psi,\chi}$ is an isometry from~${\mathcal H}_{\varpi}$ onto
$L^2(I)$.  Therefore, if $\{e_n\}_{n\in{\mathbb Z}}$ is an orthonormal
basis of $L^2(I)$, one gets an orthonormal basis
$\{\gamma_n\}_{n\in{\mathbb Z}}$ of ${\mathcal H}_{\varpi}$ in the
following way:
\begin{equation}
\widehat{\gamma}_n(\omega) =
e_n\circ\psi(\omega)/\chi\bigl(\psi(\omega)\bigr) =
\sqrt{\frac{\psi'(\omega)}{\varpi(\omega)}}\,e_n\bigl(\psi(\omega)\bigr).
\end{equation}
\bigskip

As a consequence, any $X\in{\mathcal H}_{\varpi}$, can be expanded
in~${\mathcal H}_{\varpi}$ as the sum
$$X = \sum_{k\in{\mathbb Z}}\langle X,\gamma_k\rangle_{{\mathcal
H}_{\varpi}}\,\gamma_k.$$

In this decomposition, the coefficients $\langle
f,\gamma_k\rangle_{{\mathcal H}_{\varpi}}$ are expressed in terms of
the inner product of ${\mathcal H}_{\varpi}$. However, in practice, it
is useful to have an expression of these coefficients as a
distributional duality product (so that we merely need to know $X$ and
not ${\widehat X}$).

As ${\mathcal H}_{1/\varpi}$ is isometrically isomorphic to the dual
of ${\mathcal H}_{\varpi}$, one has
\begin{equation}\label{dual}
X = 2\pi\sum_{n\in{\mathbb Z}} \langle X,
\widetilde{\gamma}_n\rangle\, \gamma_{n},
\end{equation}
where $\displaystyle \widehat{\widetilde{\gamma}}_n = \varpi
\,\widehat{\gamma}_n = \sqrt{\varpi \,\psi'} \,e_n\circ\psi$, and the
scalar product is the usual pairing of functions and
distributions. One can notice that the $\widetilde{\gamma}_n$ arise
from the preceding construction with the same function~$\psi$ but
with~$\varphi'/\chi$ instead of~$\chi$.

\section{The case $\psi=\arctan$}\label{arctan}

In this case we choose for $\chi$ the function
$(1+\tan^2)^{\frac{1+\alpha}{2}}$, where $\alpha$ is a real
parameter. The corresponding warping operator will be denoted by
${\mathcal T}_{\alpha}$. Then $I=(-\pi/2,\pi/2)$ and $\varpi(\omega) =
\chi(\arctan \omega)^2/(1+\omega^2) = (1+\omega^2)^{\alpha}$.  It
results that the space ${\mathcal H}_{\varpi}$ is the ordinary Sobolev
space ${\mathcal H}^{\alpha}({\mathbb R})$.

\subsection{A family of orthonormal bases}

Choose a real parameter~$\beta$. Then the system of functions $e_n(u)
= \e{-2\mi n\,u}/\sqrt{\pi}$\ \ (for $n\in\beta+{\mathbb Z}$) is an
orthonormal basis of $L^2(I)$. The corresponding $\gamma_n$ are
defined by
\begin{equation*}
\widehat{\gamma}_n(\omega) = \frac{\e{-2\mi n\arctan
\omega}}{\sqrt{\pi}\,(1+\omega^2)^{\frac{1+\alpha}{2}}}.
\end{equation*}

We draw the reader's attention on the unusual labeling of these
bases. But this notation will prove to be convenient.

In other terms
\begin{eqnarray}
\widehat{\gamma}_n(\omega) &=&
\frac{1}{\sqrt{\pi}\,(1+\omega^2)^{\frac{1+\alpha}{2}}}
\left[\frac{1-\mi\omega}{1+\mi\omega}\right]^{n}\\
&=& \frac{1}{\sqrt{\pi}\,(1-\mi\omega)^{\frac{1+\alpha}{2}-n}
(1+\mi\omega)^{\frac{1+\alpha}{2}+n}},
\end{eqnarray}
where we use the principal determination of $\log (1+z)$, i.e., the
one which is defined on the simply connected open set ${\mathbb
  C}\setminus \{x\in{\mathbb R} \mid x\le -1\}$ and which assumes the
value~0 at $z=0$.

\begin{remark}\label{gammatilde}
The corresponding~$\widetilde{\gamma}_n$ are defined by
\begin{equation*}
\widehat{\widetilde{\gamma}}_n(\omega) = \frac{1}{\sqrt{\pi}}\,
(1-\mi\omega)^{\frac{\alpha-1}{2}+n}\,
(1+\mi\omega)^{\frac{\alpha-1}{2}-n}.
\end{equation*}
Indeed, they correspond to~$-\alpha$ instead of~$\alpha$.
\end{remark}

The $\gamma_n$ always satisfy a recursion relation.

\begin{lemma}\label{RecFormula}
For any $\alpha\in{\mathbb R}$ and any $n\in \beta+{\mathbb Z}$, we
have
\begin{equation}\label{recursion}
\left(\frac{\alpha+1}{2}+n\right)\gamma_{n+1}(t)+2(n-t)\,\gamma_n(t) +
\left(n-\frac{\alpha+1}{2}\right)\gamma_{n-1}(t)=0.
\end{equation}
\end{lemma}

\begin{proof}
We have
$$\sqrt{\pi}\,{\widehat \gamma_n}(\omega)=\frac{\e{-2\mi n\arctan
\omega}}{(1+\omega^2)^{\frac {\alpha+1} 2}}$$ and
$$\sqrt{\pi}\,{\widehat \gamma_n}'(\omega)=-\frac{(\alpha+1)\omega
\e{-2\mi n\arctan \omega}}{(1+\omega^2)^{\frac {\alpha+3} 2}}
-\frac{2in\e{-2\mi n\arctan \omega}}{(1+\omega^2)^{\frac{\alpha+3}2}}.$$
But
$$\sqrt{\pi}\, \bigl(\widehat \gamma_{n+1}(\omega)+
\widehat\gamma_{n-1}(\omega)\bigr) = \frac{2\e{-2\mi n\arctan
\omega}}{(1+\omega^2)^\frac{\alpha+1}{2}}\left ( \frac
2{1+\omega^2}-1\right )$$
and
$$\sqrt{\pi}\,\bigl(\widehat\gamma_{n+1}(\omega)-\widehat
\gamma_{n-1}(\omega) \bigr)= -\frac{4\mi\omega \e{-2\mi n\arctan
\omega}}{(1+\omega^2)^{\frac {\alpha+3}2}}.$$
Therefore
$${\widehat \gamma_n}'(\omega)=-\frac{\mi(\alpha+1)}{4} \bigl(\widehat
\gamma_{n+1}(\omega)-\widehat \gamma_{n-1}(\omega)\bigr) -
\frac{\mi n}{2} \bigl(\widehat \gamma_{n+1}(\omega)+\widehat
\gamma_{n-1}(\omega)+2{\widehat \gamma_n}(\omega)\bigr).$$
We conclude by taking the inverse Fourier transform.
\end{proof}

Formula~(\ref{recursion}) is reminiscent of the recursion relations
satisfied by orthogonal polynomials. Indeed, as we shall see it in the
next section, Laguerre polynomials come in when $\alpha$ is an
integer. Now, we express the $\gamma_n$ in terms of the Whittaker
function.\medskip

We have to compute the inverse Fourier transform $w_{p,q}$ of
expressions of the form $(1+\mi\omega)^{-p}(1-\mi\omega)^{-q}$, where $p$
and $q$ are two real parameters.

When $p+q>1$, $w_{p,q}$ is a function, otherwise we have to deal with
a distribution.

It turns out that $w_{p,q}$ is expressible in terms of the Whittaker
function~$W$ when $p+q>1$: according to~\cite{gradshteyn}, Formula~9
on page~345, we have
\begin{multline}\label{W}
w_{p,q}(t) = \\\frac{1}{2\pi}\int_{\mathbb R}
\frac{\e{\mi t\omega}\,\dif\omega}{(1+\mi\omega)^{p}(1-\mi\omega)^{q}} =
\begin{cases}\displaystyle
\frac{t^{\frac{p+q}{2}-1}}{2^{\frac{p+q}{2}}\,\Gamma(p)}\,
W_{\frac{p-q}{2},\frac{1-p-q}{2}}(2t) & \text{if}\ t>0,\\\ &\ \\
\frac{(-t)^{\frac{p+q}{2}-1}}{2^{\frac{p+q}{2}}\,\Gamma(q)}\,
W_{\frac{q-p}{2},\frac{1-p-q}{2}}(-2t) & \text{if}\ t<0,
\end{cases}
\end{multline}
valid for $p+q>1$.

\begin{proposition}\label{whittaker} Provided that $\alpha>0$, one has
\begin{equation*}
\begin{split}
\gamma_n(t) =
\frac{|t|^{\frac{\alpha-1}{2}}}{2^\frac{\alpha+1}{2}\sqrt{\pi}}\,
\left[\frac{1}{\Gamma(\frac{1+\alpha}{2}+n)}\,
W_{n,-\alpha/2}(2t)\,H(t)\right.\\ \left.{} +
\frac{1}{\Gamma(\frac{1+\alpha}{2}-n)}\,W_{-n,-\alpha/2}(-2t)\,H(-t)
\right].
\end{split}
\end{equation*}
\end{proposition}

\proof This is a reformulation of Formula~\eqref{W}

Of course, in the distribution sense, $w_{p-1,q-1} =
w_{p,q}-w_{p,q}''$. By using this remark
and~Proposition(\ref{whittaker}), one is able to deal with the
case~$\alpha\le0$.

\subsection{The case when $\boldsymbol{\alpha}$ is an integer and
  $\boldsymbol{\beta=(\alpha+1)/2}$}\ 
\bigskip

The reason to take  $\displaystyle \beta=\frac{\alpha+1}{2}$ is that,
with this setting, both  $\displaystyle \frac{\alpha+1}{2}-n$
and $\displaystyle \frac{\alpha+1}{2}+n$ are integers when 
$n\in \beta+{\mathbb Z}$. It results that
we have to consider Formula~\eqref{W} when $p$ and $q$ are
integers. In this case the expression of $w_{p,q}$ is much simpler.

\begin{lemma}\label{Residu} Let $p$ be an integer and $q\in {\mathbb R}
 $ such that $p+q\ge 1$. Then
\begin{enumerate}
\item For $t>0$,
\begin{equation}\label{residu}
w_{p,q}(t) =
\begin{cases}\displaystyle
\frac{2\e{-t}}{2^{p+q}} \sum_{k=0}^{p-1} (-1)^{p-1-k}\binom{-q}{p-k-1}\,
\frac{(2t)^k}{k!} & \text{if\ }p>0,\\ 0&\text{otherwise}.
\end{cases}
\end{equation}
\item If moreover $q$ is a nonpositive integer, then $w_{p,q}(t)=0$
  for $t<0$.
\end{enumerate}
\end{lemma}

\proof Formula~\eqref{residu} is obtained by contour integration
in the upper half plane, the second assertion by contour integral in
the lower half plane.\medskip

 In this setting we have the following facts.

\begin{lemma}\label{symmetry}
For $n\in\frac{\alpha+1}{2}+{\mathbb Z}$, one has $\gamma_{-n}(t) =
\gamma_n(-t)$.
\end{lemma}

\begin{lemma}\label{poly}Suppose~$\alpha\ge0$ and~$t>0$.Then
\begin{enumerate}
\item if $n\le -\frac{\alpha+1}{2}$, $\gamma_n(t)=0$,
\item if $n> -\frac{\alpha+1}{2}$,
\begin{eqnarray}
\gamma_n(t) &=& \frac{\e{-t}}{2^{\alpha}\sqrt{\pi}}
\sum_{k=0}^{n+\frac{\alpha-1}{2}}
(-1)^{n+\frac{\alpha-1}{2}-k}\binom{n-\frac{\alpha+1}{2}}
    {n+\frac{\alpha-1}{2}-k} \frac {(2t)^k} {k!}\label{gamma1}\\
&=& \frac{\e{-t}}{2^{\alpha}\sqrt{\pi}}
    \sum_{k=0}^{n+\frac{\alpha-1}{2}}
    \binom{\alpha-1-k}{\frac{\alpha-1}{2}+n-k} \frac {(2t)^k}
          {k!}.\label{gamma2}
\end{eqnarray}
\end{enumerate}
\end{lemma}

\proof This is a mere rewriting of~(\ref{residu}) (notice that the
corresponding $p$ and $q$ satisfy $p+q = \alpha+1\ge 1$).

\begin{proposition}[The case $\alpha\ge0$]\label{positivalpha} Suppose
$\alpha\ge0$ and $n\in \frac{\alpha+1}{2}+{\mathbb Z}$.
\begin{itemize}
\item If $n \ge \frac{\alpha+1}{2}$,
\begin{equation}
\gamma_{\pm n}(t) =
\frac{(-1)^{n-\frac{\alpha+1}{2}}|t|^\alpha\e{-|t|}}{\sqrt{\pi}} \,
\frac{\bigl( n-\frac{\alpha+1}{2}\bigr)!}
{\bigl(n+\frac{\alpha-1}{2}\bigr)!}  \,
L_{n-\frac{\alpha+1}{2}}^{(\alpha)}(2|t|)\, H(\pm t),
\end{equation}
where $L_j^{(\alpha)}$ stands for the $j$-th Laguerre polynomial of
order~$\alpha$.

\item If $|n| < \frac{\alpha+1}{2}$,
\begin{multline}
\quad \gamma_n(t) = \frac{\e{-|t|}}{2^{\alpha}\sqrt{\pi}}\left[ H(t)
\sum_{k=0}^{\frac{\alpha-1}{2}+n}
\binom{\alpha-1-k}{\frac{\alpha-1}{2}+n-k} \frac {(2t)^k}
{k!}\right.\\
\left. {}+ H(-t) \sum_{k=0}^{\frac{\alpha-1}{2}-n}
\binom{\alpha-1-k}{\frac{\alpha-1}{2}-n-k} \frac {(-2t)^k}
{k!}\right]. \quad
\end{multline}
\end{itemize}
\end{proposition}

\begin{proof} By using Formula~\eqref{gamma1} with $p=(\alpha+1)/2+n$ and
  $q=(\alpha+1)/2-n$ one gets, for $t>0$ and $n\ge (\alpha+)/2$,
\begin{equation*}
\gamma_n(t) = \frac{\e{-t}}{2^\alpha\sqrt{\pi}}
\sum_{k=0}^{n+\frac{\alpha-1}{2}} (-1)^{n+\frac{\alpha-1}{2}}
\binom{n-\frac{\alpha+1}{2}}{n+\frac{\alpha-1}{2}-k}
\,\frac{(-2t)^k}{k!}.
\end{equation*}
The terms of this sum corresponding to $k<\alpha$ being nul, one gets
\begin{eqnarray*}
\gamma_n(t) &=& \frac{(-1)^{n-\frac{\alpha+1}
    {2}}t^\alpha\e{-t}}{\sqrt{\pi}} \sum_{k=0}^{n-\frac{\alpha+1}{2}}
\binom{n-\frac{\alpha+1}{2}}{n-\frac{\alpha+1}{2}-k}\,
\frac{(-2t)^{k}}{(k+\alpha)!}\\
&=& \frac{(-1)^{n-\frac{\alpha+1} {2}}t^\alpha\e{-t}}{\sqrt{\pi}}
\frac{(n-\frac{\alpha+1}{2})!}{(n+\frac{\alpha-1}{2})!}
\sum_{k=0}^{n-\frac{a+1}{2}}
\binom{n+\frac{\alpha-1}{2}}{n-\frac{\alpha+1}{2}-k}\,
\frac{(-2t)^k}{k!}\\
&=& \frac{(-1)^{n-\frac{\alpha+1} {2}}t^\alpha\e{-t}}{\sqrt{\pi}}
\frac{(n-\frac{\alpha+1}{2})!}{(n+\frac{\alpha-1}{2})!}
L_{n-\frac{\alpha+1}{2}}^{(\alpha)}(2t).
\end{eqnarray*}

Moreover, since $q\le0$, one has $\gamma_n(t)=0$ for $t<0$. The case
$n\le -(\alpha+1)/2$ results from $\gamma_{-n}(t) = \gamma_n(-t)$.

The second assertion results from Formula~\eqref{gamma2}

\end{proof}

\begin{proposition}[The case $\alpha\le -1$]\label{negativalpha}
Suppose $\alpha\le-1$.
\begin{itemize}
\item If $|n|<\frac{|\alpha|+1} 2$, then
$$\gamma_n(t)=\frac{1}{\sqrt{\pi}}
\,(\delta+\delta')^{-\frac{\alpha+1}{2}-n}*
(\delta-\delta')^{-\frac{\alpha+1}{2}+n},$$
where powers are iterated convolutions.
\item If $n\geq \frac{|\alpha|+1} 2$, then
\begin{equation*}
\begin{split}
\gamma_{\pm n}(t)= {}&
(-1)^{n+\frac{\alpha-1}{2}}\,\frac{2^{-\alpha}}{\sqrt{\pi}}\,\e{-|t|}
L_{n+\frac{\alpha-1}{2}}^{(-\alpha)}(2|t|)\,H(\pm t)\\
&+\frac{(-1)^{n+\frac{\alpha+1}{2}}}{2^{\alpha+1}\sqrt{\pi}}
\sum_{j=0}^{-(\alpha+1)} (-2)^{-j}
\binom{n-\frac{\alpha+1}{2}}{-(\alpha+1+j}\,(\delta\pm\delta')^{j}.
\end{split}
\end{equation*}
\end{itemize}
\end{proposition}

\begin{proof} As previously, we have to compute $\gamma_n$ knowing its
Fourier transform
\begin{equation*}
\widehat{\gamma}_n(\omega)=\frac{1} {\sqrt{\pi}\,
(1+\mi\omega)^{\frac{\alpha+1}{2}+n}
(1-\mi\omega)^{\frac{\alpha+1}{2}-n}}.
\end{equation*}

\begin{itemize}
\item If $|n|<\frac{|\alpha|+1}{2}$, then

\begin{equation*}
\sqrt{\pi}\,{\widehat \gamma_n}(\omega) =
(1+\mi\omega)^{-\frac{\alpha+1}{2}-n}(1-\mi\omega)^{-\frac{\alpha+1}{2}+n},
\end{equation*}
which gives the first assertion.\\

\item If $n\geq \frac{|\alpha|+1} 2$, then $n\geq \frac{\alpha+1}{2}$
and one writes
\begin{eqnarray}
\sqrt{\pi}\,{\widehat \gamma_n}(\omega)&=&\frac
{\bigl(2-(1+\mi\omega)\bigr)^{n-\frac{\alpha+1}{2}}}
{(1+\mi\omega)^{n+\frac{\alpha+1}{2}}}\nonumber\\
&=&\sum_{j=0}^{n-\frac{\alpha+1}{2}}\binom{n-\frac{\alpha+1}{2}}{j}
\frac{2^j(-1)^{n-\frac{\alpha+1}{2}-j}}
{(1+\mi\omega)^{j+\alpha+1}}.\label{Int1}
\end{eqnarray}

We split the sum in equation (\ref{Int1}) into two sums, the first one
being the sum for $j=0$ to $-(\alpha+1)$,
$$\sum_{j=0}^{-(1+\alpha)} \binom{n-\frac{\alpha+1}{2}}{j} 2^j
(-1)^{n-\frac{\alpha+1}{2}-j} (1+\mi\omega)^{-(1+\alpha+j)}$$
which can be rewritten as
$$\frac{(-1)^{n+\frac{\alpha+1}{2}}}{2^{\alpha+1}}
\sum_{j=0}^{-(1+\alpha)} \binom{n-\frac{\alpha+1}{2}}{-(\alpha+1+j)}
(-2)^{-j} (1+\mi\omega)^j.$$
and gives the singular component.

The second sum
$$\sum_{j=-\alpha}^{n-\frac{\alpha+1}{2}}
\binom{n-\frac{\alpha+1}{2}}{j} 2^j(-1)^{n-\frac{\alpha+1}{2}-j}
(1+\mi\omega)^{-(1+\alpha+j)}$$
is the Fourier transform of
$$\e{-t}H(t)\sum_{j=-\alpha}^{n-\frac{\alpha+1}{2}}
\binom{n-\frac{\alpha+1}{2}}{j} 2^j
(-1)^{n-\frac{\alpha+1}{2}-j}\frac{t^{j+\alpha}}{(j+\alpha)!},$$
that is of
$$(-1)^{n+\frac{\alpha-1}{2}}2^{-\alpha}\e{-t}H(t)
\sum_{k=0}^{n-\frac{\alpha+1}{2}+\alpha}
\binom{n-\frac{\alpha+1}{2}+\alpha-\alpha}
{n-\frac{\alpha+1}{2}+\alpha-k}\frac{(-2t)^k}{k!},$$
or of
$$(-1)^{n+\frac{\alpha-1}{2}}2^{-\alpha}
\e{-t}H(t)\,L_{n+\frac{\alpha-1}{2}}^{-\alpha}(2t).$$ 

\item Le case $n\le -\frac{|\alpha|+1} 2$ is handled by
Lemma~\ref{symmetry}.
\end{itemize}
\end{proof}

\subsection{Decomposition of Sobolev spaces}

In this section, $\alpha$ is a positive integer.

Proposition~\ref{positivalpha} shows that we have to deal with three
types of basis functions:
\begin{itemize}
\item the functions $\gamma_n$ with $n\leq -\frac{\alpha+1}{2}$, that
are supported in ${\mathbb R}_{-}.$,
\item the functions $\gamma_n$ with $|n|<\frac{\alpha+1}{2}$, that are
two-sided functions, localized around the origin, and
\item the functions $\gamma_n$ with $n\geq \frac{\alpha+1}{2}$, that
are supported in ${\mathbb R}_+$.
\end{itemize}

Let ${\mathcal H}_-^\alpha$, ${\mathcal H}_\circ^\alpha$, and
${\mathcal H}_+^\alpha$ stand for the closed subspaces of ${\mathcal
H}^{\alpha}$ generated respectively by $\{\gamma_n\}_{n\le
-\frac{\alpha+1}{2}}$, $\{\gamma_n\}_{|n|<\frac{\alpha+1}{2}}$, and
$\{\gamma_n\}_{n\ge \frac{\alpha+1}{2}}$.

${\mathcal H}_+^\alpha$ and ${\mathcal H}_-^\alpha$ are the sets of
elements of ${\mathcal H}^\alpha$ which are supported in $[0,+\infty)$
and $(-\infty,0]$ respectively. The orthogonal complement ${\mathcal
H}_\circ^\alpha$ of ${\mathcal H}_+^\alpha \oplus {\mathcal
H}_-^\alpha$ is of dimension~$\alpha$.

Due to Proposition~\ref{negativalpha} The elements
$\widetilde{\gamma}_n$ for $|n|<\frac{\alpha+1}{2}$, are linear
combinations of $\delta$, $\delta'$,\dots, $\delta^{(\alpha-1)}$. This
means that the projection $X_{\circ}$ of $X\in{\mathcal H}^{\alpha}$
on ${\mathcal H}_\circ^{\alpha}$ can be expressed as
\begin{equation*}
X(t) = \sum_{j=0}^{\alpha-1} X^{(j)}(0)\, Y_j(t),
\end{equation*}
where the functions~$Y_j$ are linear combinations of the $\gamma_n$
for $|n|<\frac{\alpha+1}{2}$.

In these conditions, it is not difficult to show that
\begin{equation*}
Y_j(t) = \frac{t^j}{j!}\,\e{-|t|}\sum_{k=0}^{\alpha-j}
\frac{|t|^k}{k!}.
\end{equation*}

The following figure shows the $Y$ functions when $\alpha=5$.

\begin{figure}[htbp]
\begin{center}
\includegraphics[width=12cm]{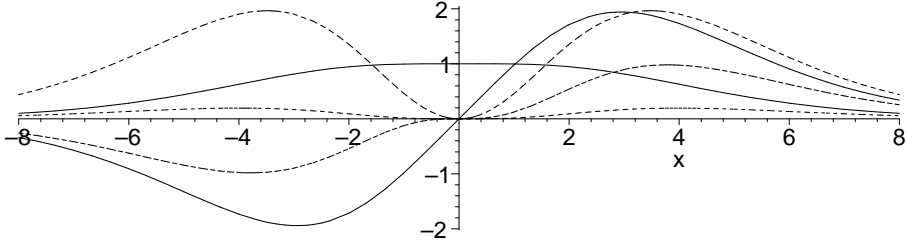}
\caption{The Y functions for $\alpha=5$}\label{middle}
\end{center}
\end{figure}

If $X_+$ is the projection of $X$ on ${\mathcal H}_+^{\alpha}$, one
has (see Equation\eqref{dual} and Remark~\ref{gammatilde})
\begin{equation*}
X_+(t) = 2\pi \sum_{n\ge\frac{\alpha+1}{2}} a_n\, \gamma_n(t),
\end{equation*}
where
\begin{equation*}
a_n = \frac{(-1)^{n-\frac{\alpha+1}{2}}2^{\alpha-1}}{\sqrt{\pi}}
\int_{0}^{+\infty} X_+(t)\, L_{n-\frac{\alpha+1}{2}}^{(\alpha)}(2t)\,
\e{-t}\, \dif t,
\end{equation*}
due to Proposition~\ref{negativalpha}.

If $0\le k<\alpha$ and $n\ge \frac{\alpha+1}{2}$, one has $\langle
Y_k,\gamma_{\pm n}\rangle =0$, from which one deduces
\begin{equation*}
\int_{0}^{+\infty} Y_k(t)\, L_{n-\frac{\alpha+1}{2}}^{(\alpha)}(2t)\,
\e{-t}\, \dif t = -\frac{1}{2}(\pm1)^k\sum_{j=k}^{\alpha-1} (-2)^{-j}
\binom{n+\frac{\alpha-1}{2}}{\alpha-1-j}\, \binom{j}{k}.
\end{equation*}

Therefore, if $X\in{\mathcal H}^{\alpha}$, and if, for $n\ge
\frac{\alpha+1}{2}$, one defines
\begin{equation*}\begin{split}
a_{\pm n} = \frac{(-1)^{n-\frac{\alpha+1}{2}}2^{\alpha}}{\sqrt{\pi}}
\Bigg[\int_{\mathbb R} X(t)\,
L_{n-\frac{\alpha+1}{2}}^{(\alpha)}(2|t|)\, \e{-|t|}H(\pm t)\, \dif t\\
{} + \frac{1}{2}\sum_{k=0}^{\alpha-1}(\pm1)^k X^{(k)}(0)
\sum_{j=k}^{\alpha-1} (-2)^{-j}
\binom{n+\frac{\alpha-1}{2}}{\alpha-1-j}\, \binom{j}{k}\Bigg],
\end{split}\end{equation*}
then one has
\begin{equation*}
X = \sum_{j=0}^{\alpha-1} X^{(j)}(0)\, Y_j +
\sum_{n=\frac{\alpha+1}{2}}^{+\infty} \bigl(
a_n\gamma_n+a_{-n}\gamma_{-n}\bigr).
\end{equation*}

The four following figures show some functions $\gamma_n$.

\begin{figure}[htbp]
\begin{center}
\includegraphics[width=6cm]{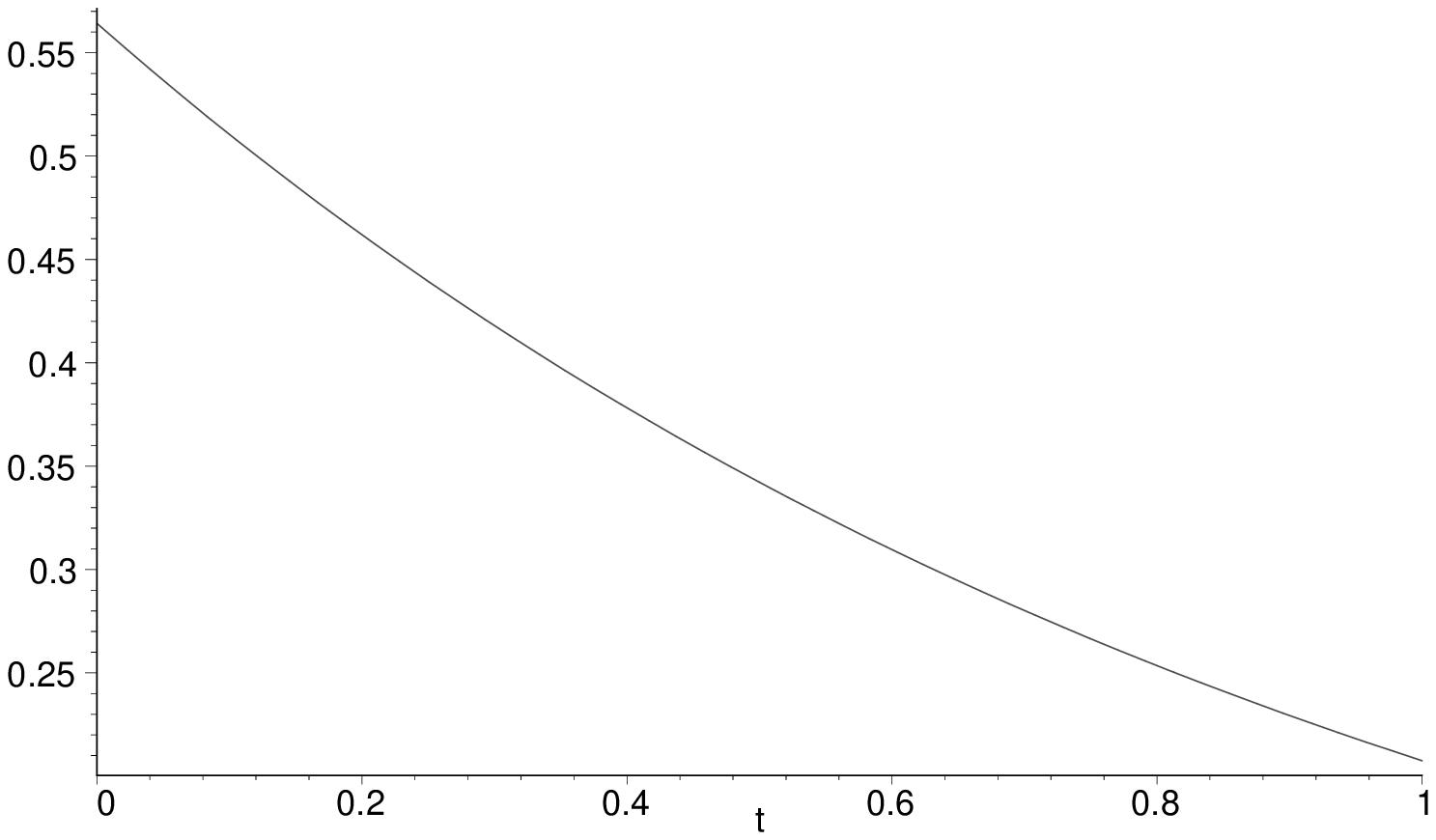}
\includegraphics[width=6cm]{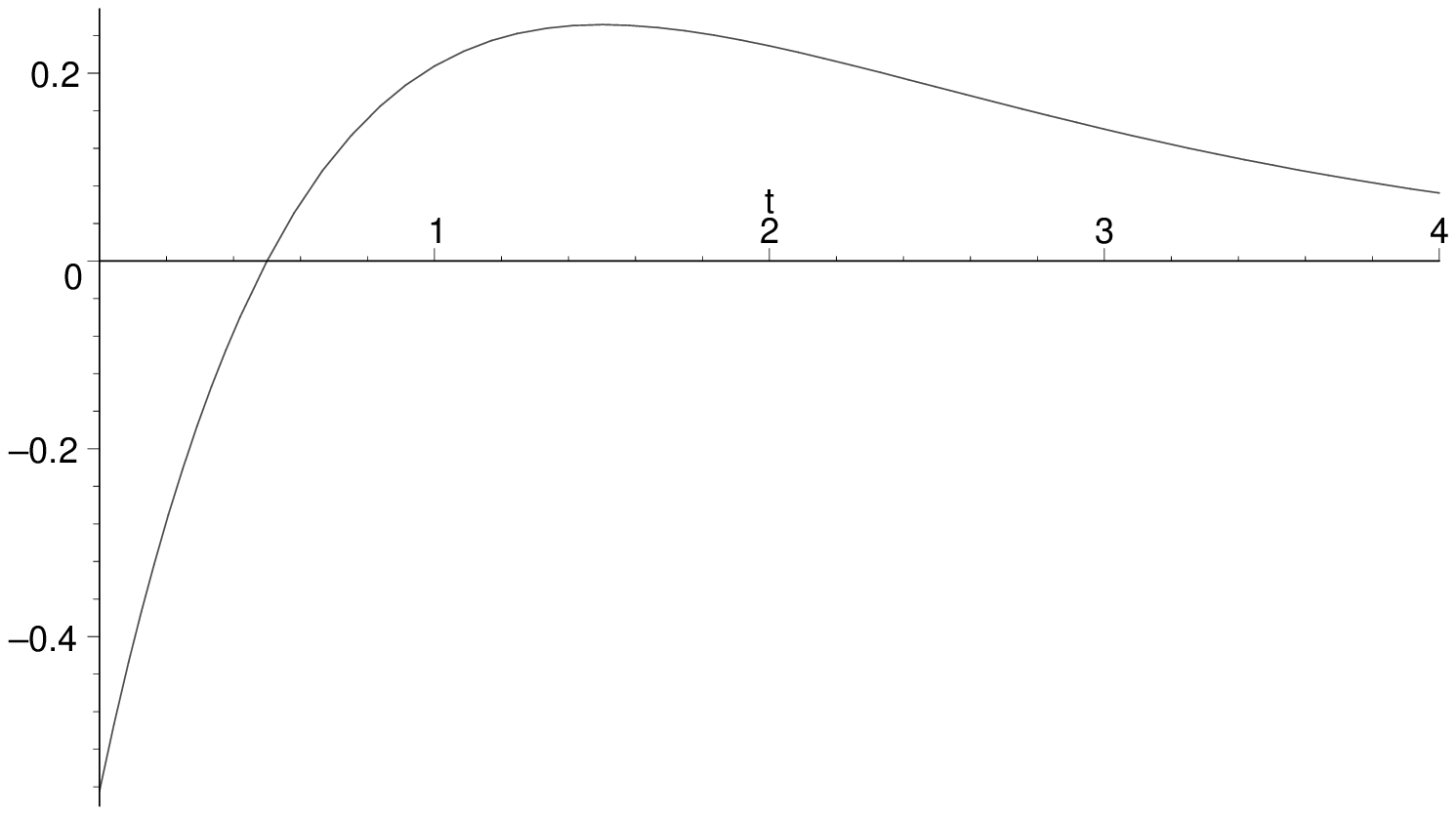}
\caption{Graphs of $\gamma_{0.5}$ and $\gamma_{1.5}$ for
$\alpha=0$}\label{gammas1}
\end{center}
\end{figure}

\begin{figure}[htbp]
\begin{center}
\includegraphics[width=6cm]{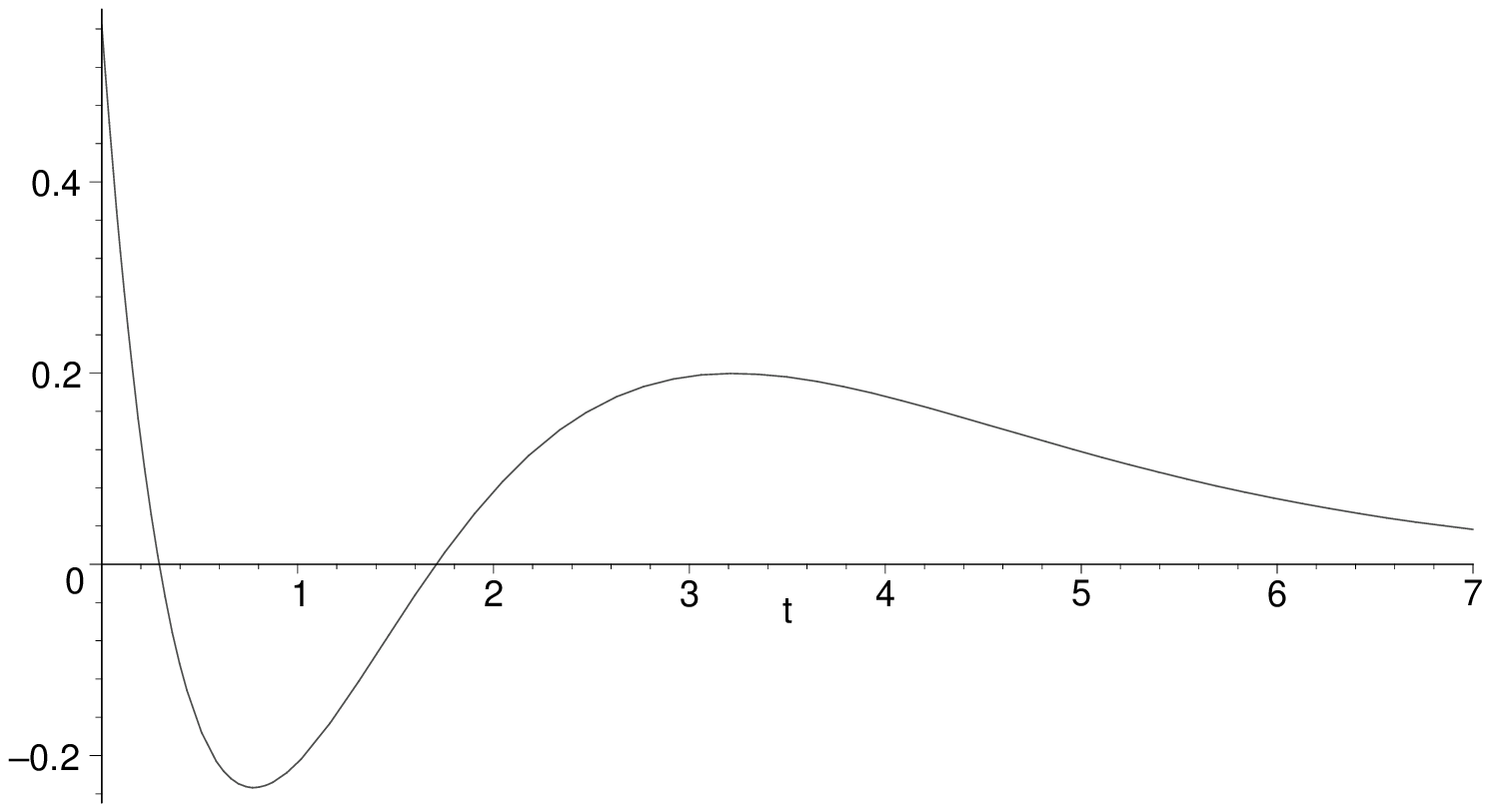}
\includegraphics[width=6cm]{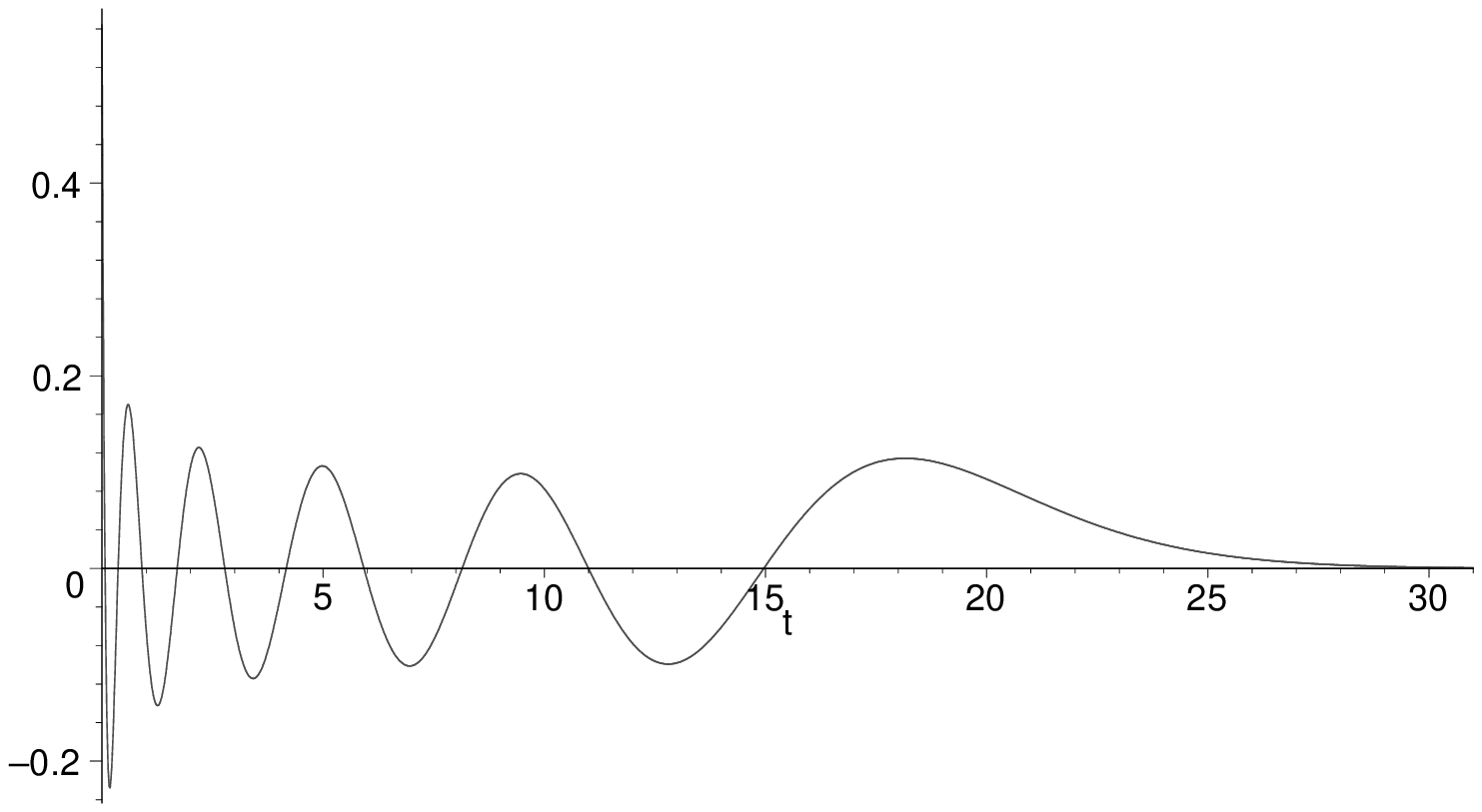}
\caption{Graphs of $\gamma_{2.5}$ and $\gamma_{10.5}$ for
$\alpha=0$}\label{gammas2}
\end{center}
\end{figure}

\begin{figure}[htbp]
\begin{center}
\includegraphics[width=6cm]{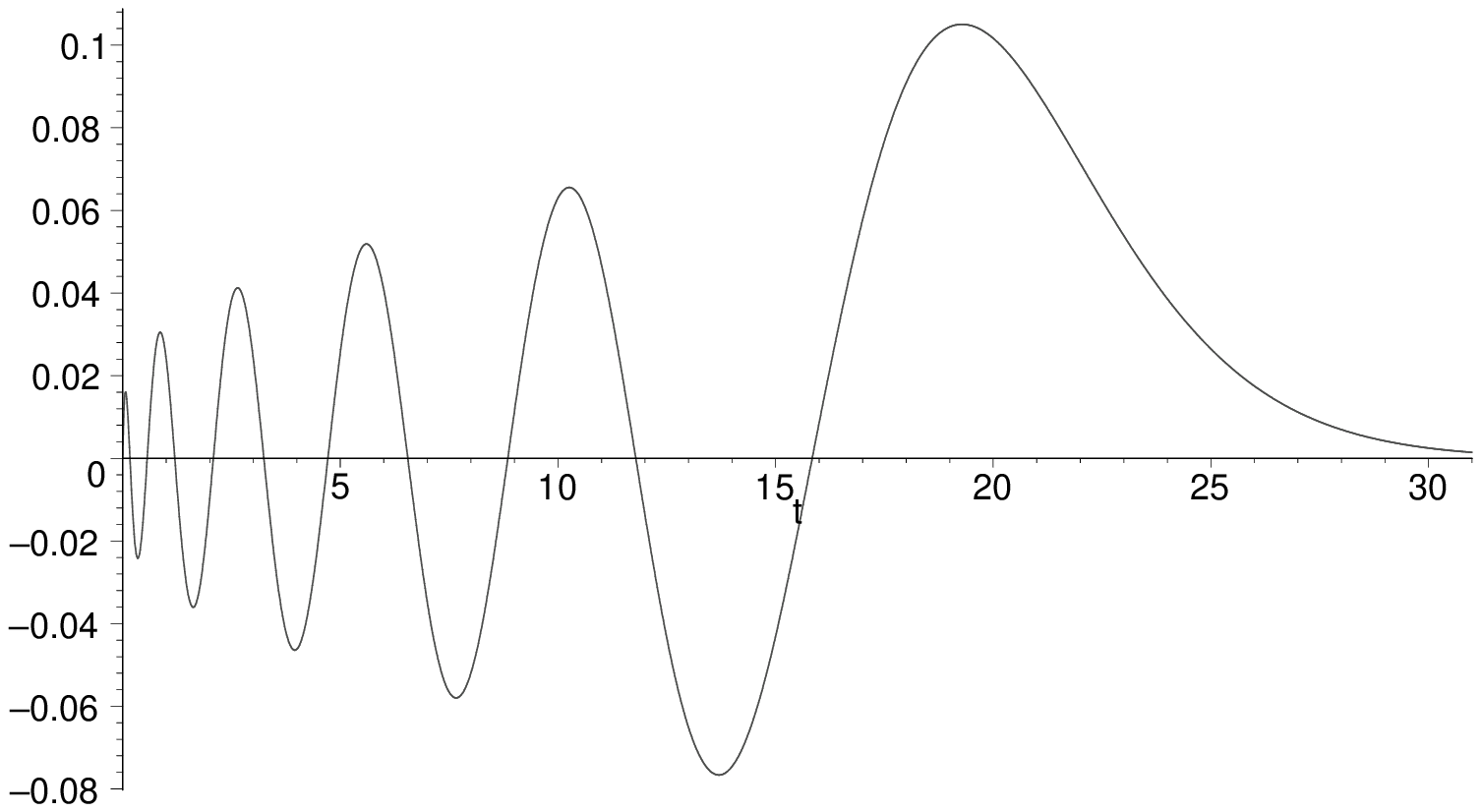}
\includegraphics[width=6cm]{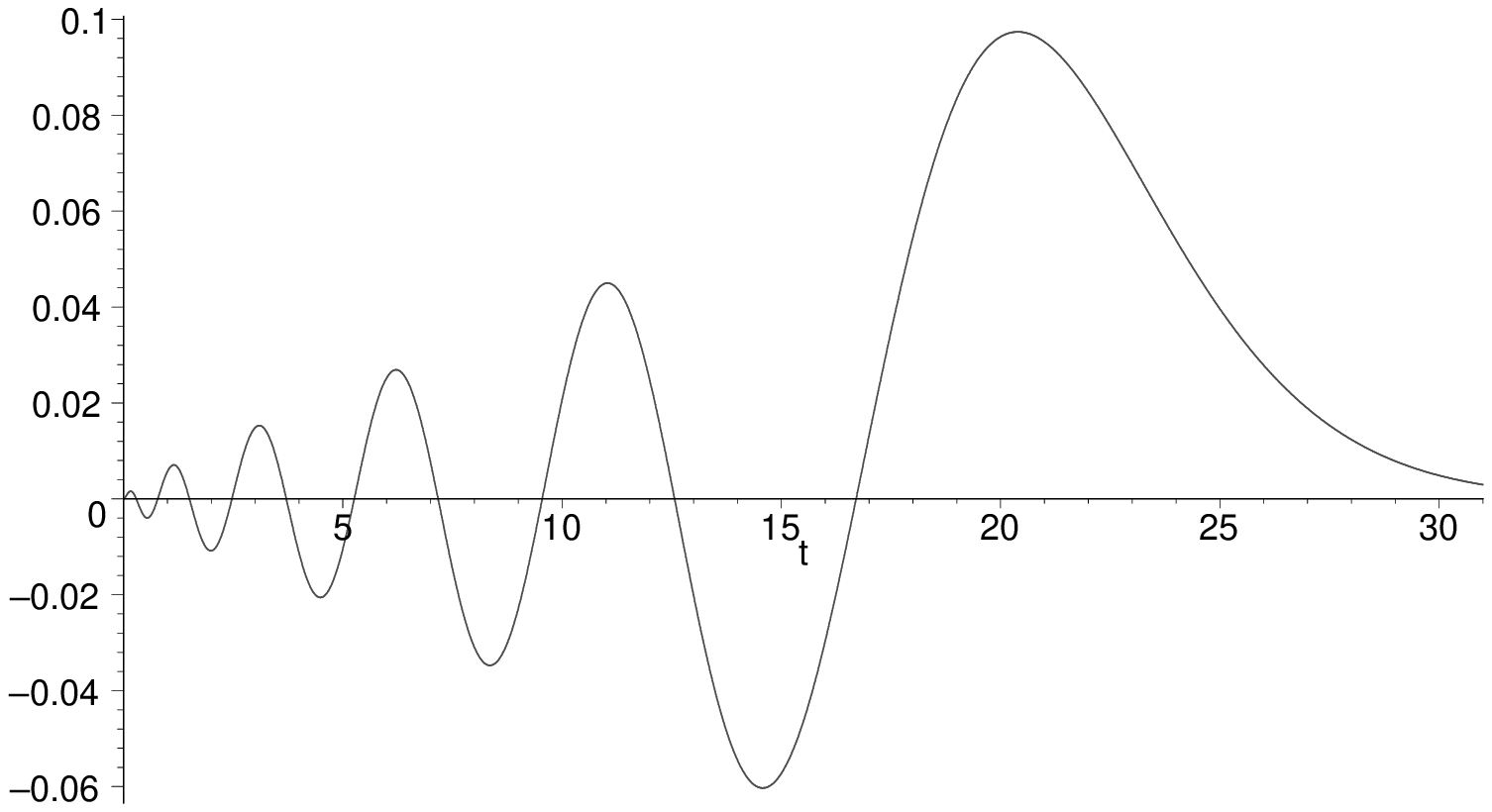}
\caption{Graphs of $\gamma_{11}$ for $\alpha=1$ and $\gamma_{11.5}$
$\alpha=2$}\label{gammas3}
\end{center}
\end{figure}

\begin{figure}[htbp]
\begin{center}
\includegraphics[width=6cm]{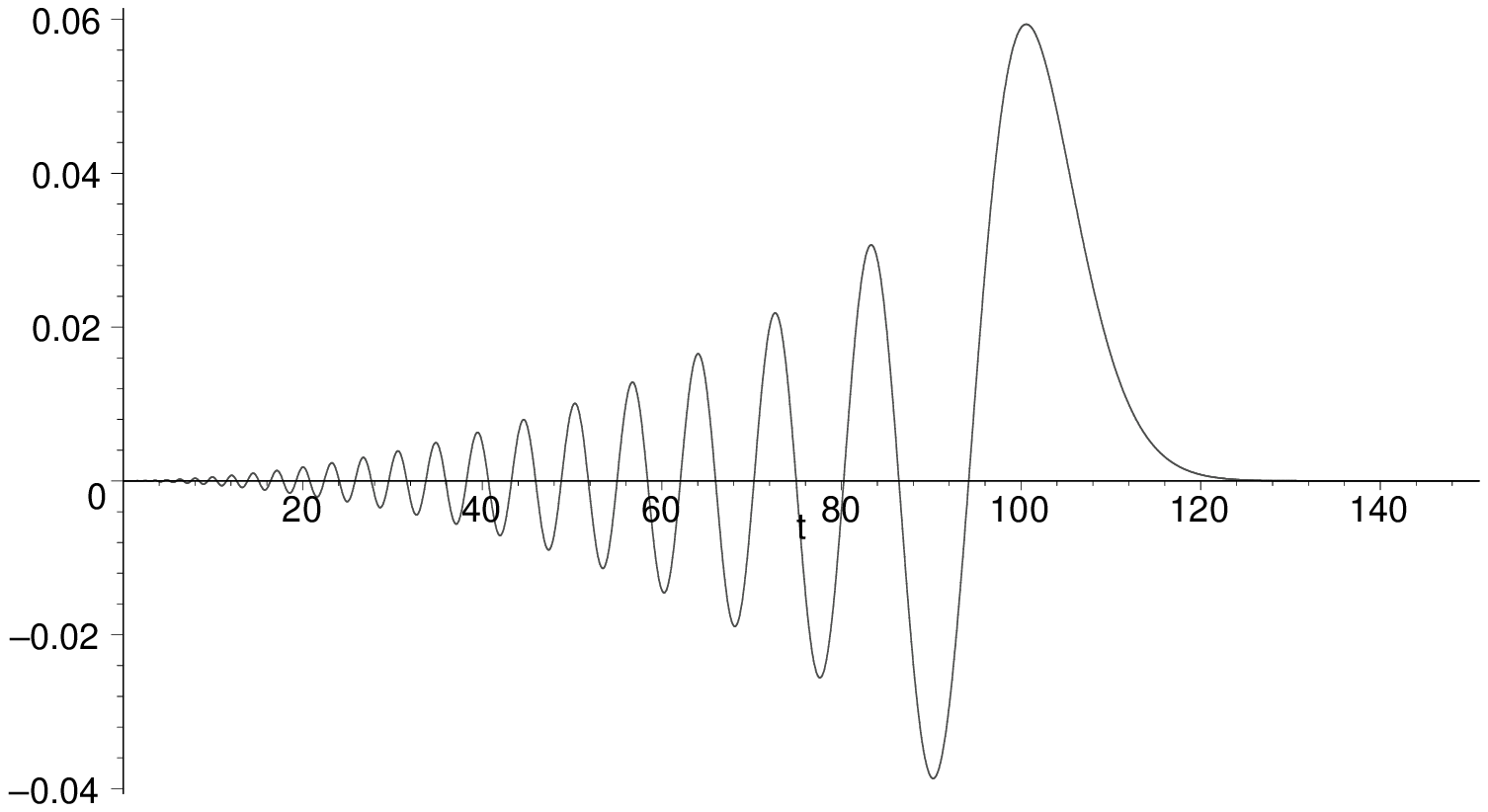}
\includegraphics[width=6cm]{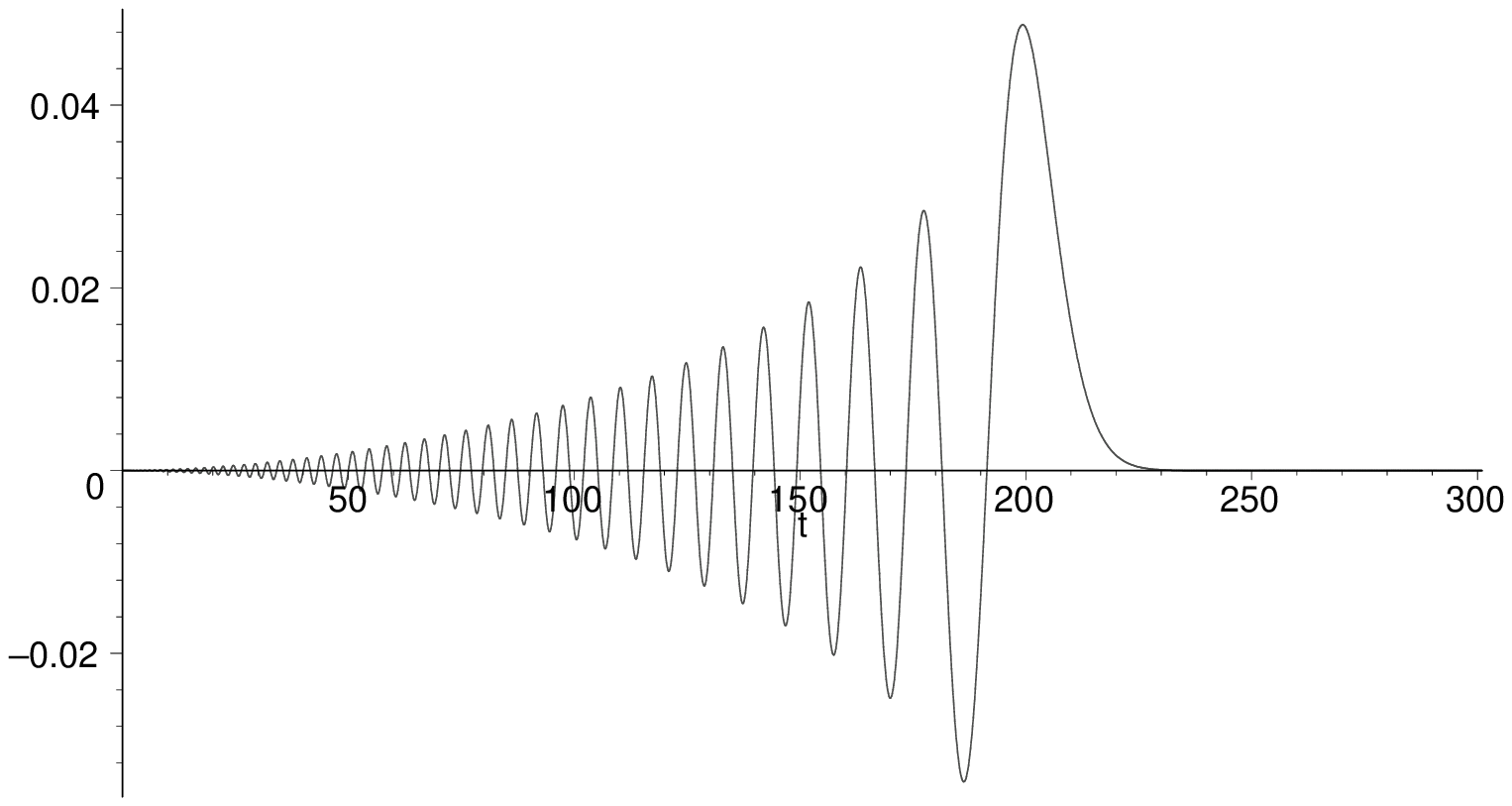}
\caption{Graphs of $\gamma_{52.5}$ and $\gamma_{102.5}$ for
$\alpha=4$}\label{gammas4}
\end{center}
\end{figure}

\section{Speed of approximation}\label{approx}

\subsection{The strategy}

We return for a moment to the general setting of
Section~\ref{setting}. We saw that a signal in ${\mathcal H}_{\varpi}$
can be written a the sum of a series $\sum_{n\in{\mathbb Z}}
a_n\,\gamma_n$, which converges in norm within this Hilbert space.

In practice, and especially in a sampling framework, it is
important both to secure convergence and to evaluate approximation
rates in various norms such as $L_{\infty}$ and $L^2$ in addition
to ${\mathcal H}_{\varpi}$. As we shall see, it is possible to
recover uniform or $L^2$ estimates from the ${\mathcal
H}_{\varpi}$-norm. Indeed, if the weight~$\varpi$ is bounded from
below by a positive constant, then the ${\mathcal H}_{\varpi}$
convergence implies the $L^2$ one and if $\int_{|\omega|>1}
\varpi(\omega)^{-1}\,\dif \omega < +\infty$, it implies uniform
convergence. Moreover, if $\int_{|\omega|>1}
\frac{|\omega|^{2k}}{\varpi(\omega)}\, \dif \omega < +\infty$, then
one has uniform convergence of the derivatives up to order~$k$.

As a consequence, estimates of the uniform norm or the $L^2$ norm
of $\mbox{err}_N = \sum_{|n|>N} a_n\gamma_n$ may be deduced from
estimates $s_N = \sum_{|n|>N} |a_n|^2$. But the coefficients $a_n$
are nothing but the Fourier coefficient of the warped function
${\mathcal T}^{\psi,\chi}(X)$ with respect to the orthonormal
basis $\{e_n\}$. Our strategy is thus simple: Choose for $\{e_n\}$
the trigonometric system, and then make use of classical
properties of the Fourier coefficients. For instance, if a $L^1$
function on the torus is H\"older~$h$ in the $L^1$ norm, one has
$a_n=\bo{|n|^{-h}}$.  In this way one is able to get the speed of
convergence of the series for functions~$X$ such that the warped
signal $\widehat{X}\circ\varphi(u)
\,\varphi'(u)^{\frac{\alpha+1}{2}}$ extends as a~$|I|$-periodic
$C^k$-function.

In order to achieve this program and to be able to give tractable
conditions on the signal under processing, one has to particularize
the warping function. The case studied in Section~\ref{arctan} gives
rise to explicit formulae, but it will appear that there are not
enough free parameters. This is why we introduce more general a family
of warping operators.

\subsection{Another family of warping operators}

Inspired by the case $\varphi(u)=\tan(u)$ treated above, we explore a
more general setting where $\psi$ and $\chi$ are given by
\begin{equation}
\psi(\xi)=c_1\int_0^\xi \frac{\dif v}{(1+v^2)^\beta}\text{\quad and\quad}
\chi = (\varphi')^{\frac{\alpha+1}{2}},
\end{equation}
where $\beta>1/2$ and $c_1^{-1} = \frac{2}{\pi}\int_{0}^{+\infty}
\frac{\dif v}{(1+v^2)^{\beta}}$ (as previously, $\varphi$ stands for the
function reciprocal to~$\psi$).

Note that if $\beta\leq\frac12$ could be of interest, but is not
studied in this work.

We have
\begin{equation}
\varphi^\prime\circ\psi(\omega) = 1/\psi'(\omega) =
c_1^{-1}(1+\omega^2)^\beta.
\end{equation}
As a consequence, ${\mathcal H_{\varpi}}$ is the usual Sobolev space
of order $\alpha\beta$
$${\mathcal H}^{\alpha\beta}=\left\{X\in\mathcal S^\prime({\mathbb
R})\text{ such that }\int_{\mathbb R}
|\widehat{X}(\omega)|^2(1+\omega^2)^{\alpha\beta}\dif f<+\infty\right\}.$$

It can be checked that, for any $k\geq 0$,
\begin{equation}\label{derest}
\frac{\dif^{k+1}\psi}{\dif\omega^{k+1}}(\omega)
\sim\frac{c_k}{\omega^{2\beta+k}}
\end{equation}
for large~$|\omega|$.

Moreover, when $\omega\to+\infty$,
\begin{equation}\label{est2}
\psi(\omega)=\frac{\pi}{2} -
\frac{c_3}{\omega^{2\beta-1}}+\lo{\frac{1}{\omega^{2\beta-1}}}.
\end{equation}
So, when $u\rightarrow\frac \pi 2$,
\begin{equation}\label{est3}
\varphi(u)\sim\left (\frac{\pi}{2}-u\right)^{-\frac{1}{2\beta-1}}.
\end{equation}
Finally, using (\ref{derest}) and (\ref{est3}), it is easily proved
that for all $k\geq 0$
\begin{equation}\label{est4}
\varphi^{(k)}(u)\sim\left
(\frac{\pi}{2}-u\right)^{-\frac{1}{2\beta-1}-k}.
\end{equation}

We are now ready to define a family of functional spaces for which
speed of convergence results are easily obtained.

\begin{definition}  
If $m \in {\mathbb N}$ and $\mu >0$, let ${\mathcal W}_m^{\mu}$ denote
the space of functions $X\in\mathcal S^\prime({\mathbb R})$ such that
\begin{enumerate}
\item $\widehat{X}$ is $m$ times differentiable
\item if $0\leq k\leq m$ then
$|\widehat{X}^{(k)}(\omega)|=\bo{|\omega|^{-(\mu+k)}}$ for large
$|\omega|$.
\end{enumerate}
\end{definition}

Condition 2 obviously entails that ${\mathcal W}_m^{\mu}$ is a subspace of
${\mathcal H}^{\mu-\frac{1}{2}-\varepsilon}$ for all
positive~$\varepsilon$. Condition~1 imposes some smoothness on
$\widehat{X}$. Functions in ${\mathcal W}_m^{\mu}$ are thus sufficiently
``well-behaved'' both in the time and frequency domains.

\begin{proposition}
Let $m$ be a positive integer and $\alpha$, $\beta$, and~$\mu$ be
nonnegative real numbers satisfying $\beta>\frac{1}{2}$,
$2\alpha\beta<2\mu-1$, and $X\in {\mathcal W}_m^\mu$. Let $\sum_{n\in{\mathbb
    Z}|} a_n\gamma_n$ be the $(\alpha,\beta)$-development of an
$X\in{\mathcal H}^{\alpha\beta}$. Then
\begin{itemize}
\item
\begin{equation*}\label{remainder}
\left\|\sum_{|n|> N} a_n\gamma_n\right\|_2 =
\bo{N^{\frac{1}{2}-\varkappa}},
\end{equation*}
\item if moreover $\alpha\beta>\frac{1}{2}$,
\begin{equation*}
\left\|\sum_{|n|\ge N} a_n\gamma_n\right\|_\infty =
\bo{N^{\frac{1}{2}-\varkappa}},
\end{equation*}
\end{itemize}
where~$\varkappa = \min\left\{m,
\frac{\mu-\alpha\beta-\beta}{2\beta-1}\right\}$.
\end{proposition}

\begin{proof} 
Since $\alpha\beta < \mu-\frac{1}{2}$, we have $X\in {\mathcal
H}^{\alpha\beta}$. Consider the warped function
$g(u)=\widehat{X}\circ\varphi(u)\varphi'(u)^{\frac{\alpha+1}{2}}$. Due
to \ref{est3} and to the fact $X\in{\mathcal H}_m^\mu$,  for all
$k\leq m$, one has,
$$g^{(k)}(u)=\bo{\left (\frac \pi 2-|u|\right)^{\frac
{\mu-\beta(\alpha+1)} {2\beta-1}-k}}$$
when $u\rightarrow\pm\frac \pi 2$. This means that $g$ extends as a
$\pi$-periodic $\Lambda_\alpha$-function (see~\cite{Zygmund}). It
results the estimate $a_n = \bo{|n|^{-\varkappa}}$ for its Fourier
coefficients. This implies $\left\| \sum_{|n|>N}
a_n\gamma_n\right\|_{{\mathcal H}^{\alpha\beta}}^2 = \sum_{|n|>N}
|a_n|^2 = \bo{N^{1-2\varkappa}}$, from which the conclusions easily
follow..
\end{proof}

\begin{theorem}\label{speed}
Let $X$ belong to ${\mathcal W}_m^\mu$. Then one has the following facts.
\begin{enumerate}
\item If $\mu>1$, there exist $\alpha$ and $\beta$ such that $||X -
X_N||_\infty = \bo{\frac{1}{N^{m-\frac{1}{2}}}}$.
\item If $\mu>1/2$, there exist $\alpha$ and $\beta$ such that $||X -
X_N||_2 = \bo{\frac{1}{N^{m-\frac{1}{2}}}}$.
\end{enumerate}
Where, in both cases, $X_N$ stands for the partial sum $\sum_{|n|\le
N} a_n\gamma_n$
of the corresponding ($\alpha,\beta$)-warping expansion of~$X$.
\end{theorem}

\begin{proof}
If suffices to check that, given $m$ and~$\mu$ fulfilling the
hypotheses, one can find $\alpha$ and~$\beta$ satisfying the
constraints in Proposition~\ref{remainder}.
\end{proof}

\begin{figure}[htbp]
\begin{center}
\includegraphics[width=11cm]{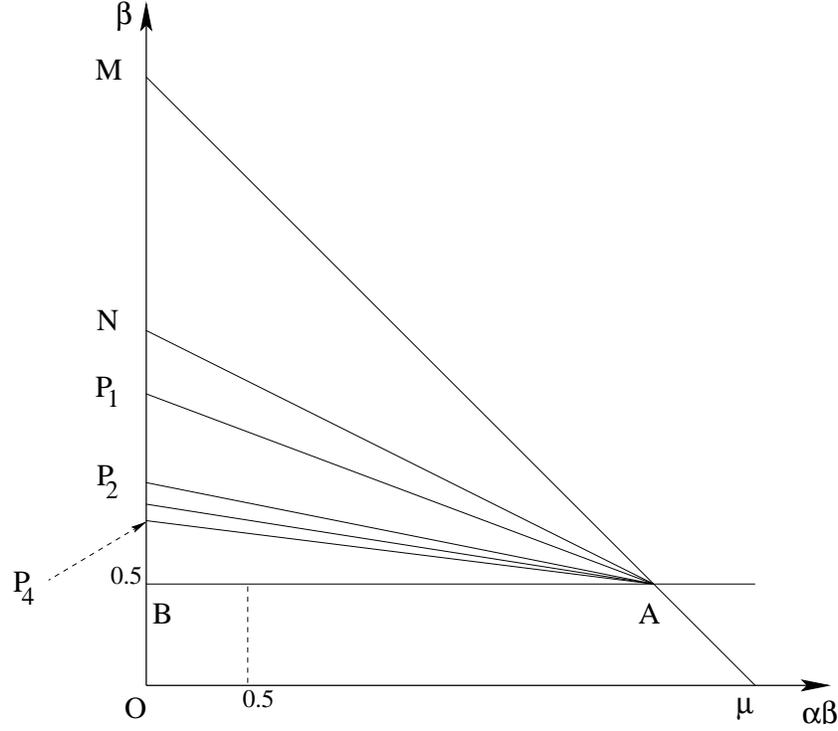}
\caption{Explanation of theorem~\ref{speed}.}\label{triangle}
\end{center}
\end{figure}

The following corollary shows that functions with very regular Fourier
transforms can be approximated with any prescribed polynomial speed.

\begin{corollary}\label{convinf}
Let $X$ belong to ${\mathcal W}_\infty^\mu$, with $\mu>1$ (rep.\ $\mu>1/2$).
Then, for all $\gamma >0$, there exists a warping operator such that
$||X - X_N||_\infty = \lo{N^{-\gamma}}$ (resp.\ $||X - X_N||_2 =
\lo{N^{-\gamma}}$).
\end{corollary}

\subsection{Comparison with WKS sampling}

\subsubsection*{Introductory remarks}

In this section, we deal with the practical aspects of our warping
method. We showed that, provided $X$ belongs to a given (large) class
of signals, excellent approximations of $X$ can be obtained by keeping
a finite number of terms in the sum $\sum_{n \in {\mathbb Z}} \langle
X,\gamma_n\rangle_{H_{\varpi}} \gamma_n$.

Now, we are going to compare the warping with the classical sampling
method based on low-pass filtering.

Assume we are given a analog signal $X$, which needs to be
digitized for purposes of storing, transmitting, or digital
processing. A crude application of Shannon sampling consists in the
following steps:
\begin{enumerate}
\item to fix a sampling frequency $\omega_0$,
\item to low-pass $X$, i.e., to compute the convolution $X_l = X*g$,
  where $\widehat{g} = {\bf 1}_{[-\omega_0/2\pi,\omega_0/2\pi]}$,
\item to approximate $X(t)$ by $\widetilde{X}_l(t) = \sum_{|n| \le
N_s} X_l(n) \sinc
\frac{t-\frac{2\pi n}{\omega_0}}{\frac{2\pi}{\omega_0}}$.
\end{enumerate}

This procedure generates two kinds of errors. The first one is due to
low-pass filtering. The second one arises in step~3, because of the
truncation of the series. Such a truncation is  unavoidable since
obviously one can access only finitely many samples in practice.
\smallskip

The warping procedure for digitizing signals proceeds as follows:

\begin{enumerate}
\item to fix a natural number $N_w$,
\item to compute the scalar products $\langle X,\gamma_k
\rangle_{H_{\varpi}}$ for $n = -N_w \ldots N_w$,
\item to approximate $X(t)$ by $\sum_{n = -N_w}^{N_w} \langle
X,\gamma_n\rangle_{H_{\varpi}} \gamma_n$.
\end{enumerate}

Again two kinds of errors are made. The first one lies in the
estimation of the scalar products. As previously, the second one is a
consequence of keeping finitely many terms in the sum in step 3. In
this work, we shall assume that the first kind of error is negligible
with respect to the second one. Indeed, when $\varphi = \tan$, we
obtained explicit expressions for the $\gamma_n$. In this case, it is
not too hard to devise a sufficiently precise numerical approximation
scheme using these expressions. In more general cases, the scalar
products may be harder to compute. We plan to investigate numerical
quadrature schemes for solving this important problem in a forthcoming
work.
\smallskip

If the WKS and warping procedures are to be compared, it is fair
to set $N_s=N_w$, henceforth denoted $N$. One then needs to select
a value for $\omega_0$. Obviously, in the case where $X$ is not
bandlimited, one would like to take $\omega_0$ as large as
possible. However, in practice, one passes from the ``time
domain'' to the ``Fourier domain'' through the Fast Fourier
Transform. As a consequence, it does not make sense to choose
$\omega_0$ larger than~$N$.

\subsubsection{Worst case comparison}

It is not easy to compare the error made in approximating an
arbitrary function through WKS sampling and warping. In order to
obtain tractable results, we shall rather compare the worst case
approximations in both situations.

More precisely, for a function $X$ in ${\mathcal W}_m^\mu$, the
$L^\infty({\mathbb R})$ error when considering $N$ terms in the
cardinal series and setting $\omega_0=N$ is of order not larger
than
$$\int_N^{+\infty}\frac{\dif \omega}{\omega^\mu}=\frac 1 {N^{\mu-1}}$$
whereas the $L^2({\mathbb R})$ error is of order not larger than
$$\frac 1 {N^{\mu-\frac12}}.$$

These errors only depend on $\mu$, which is in contrast to what
happens for the errors in the warping method.

Let us introduce the ratios of the worst cases errors
corresponding to the warping and WKS sampling:
$$\rho_\infty(N)=\frac{N^{\varkappa-\frac12}}{N^{\mu-1}} =
N^{\varkappa+\frac{1}{2}-\mu}$$
$$\rho_2(N)=\frac{N^{\varkappa-\frac12}}{N^{\mu-\frac 1 2}} =
N^{\varkappa-\mu}.$$

The following proposition gives conditions on $\alpha, \beta, m$,
and $\mu$ under which warping is preferable to low-pass filtering,
i.e., yields faster convergence rates.

\begin{proposition}[Comparison in $L_2$ and $L_{\infty}$]\label{comp}
Let $m$ be a positive integer and $\alpha$, $\beta$, and~$\mu$ be
nonnegative real numbers satisfying $\beta>\frac{1}{2}$,
$0<\alpha\beta<\mu-\frac12$, $2\mu(1-\beta)>\alpha\beta+\beta$,
and $m>\mu$. Then, the ratio of the worst cases errors for a
function in ${\mathcal W}_m^\mu$ sampled through the WKS theorem
and through the warping operator with parameters ($\alpha, \beta$)
are such that:
\begin{itemize}
\item $\lim_{N\to\infty} \rho_2 = +\infty$, and,
\item if $\alpha\beta>\frac{1}{2}$, $\lim_{N\to\infty}
\rho_\infty = +\infty$.
\end{itemize}
\end{proposition}

Note that the conditions set on the parameters imply that $\mu >1$
for the $L_\infty$ error, and that $\mu > \frac{1}{2}$ for the
$L_2$ error. Accordingly, $m$ must be at least 2 in the $L_\infty$
case, and at least 1 in the $L_2$ case: In order for the warping
procedure to be better than the WKS one in, e.g., the $L_2$ sense,
$\widehat{X}$ must thus be at least once differentiable with a
derivative decaying faster than $\frac{1}{\omega^{3/2}}$. For
instance, functions in ${\mathcal W}_1^{1/4}$ are advantageously
sampled through warping (in the $L_2$ sense) with, e.g.,
$\beta=\frac{11}{20} - \varepsilon, \alpha \beta  = \frac{1}{8}, 0
< \varepsilon < \frac{1}{20}$.

\begin{proof}
This results from Proposition~\ref{remainder}.
\end{proof}

\section{Numerical experiments}\label{numexp}

We now present results of numerical experiments that illustrate
the behavior of the warping method for approximating certain
functions. We compare it to the quality of approximation obtained
using the classical WKS sampling framework.

\subsection{Methodology}

We give ourselves a series of functions to be analyzed, and then
reconstructed using the WKS sampling algorithm, and our method. We
focus on the case $\alpha=\beta=1$, i.e., $\psi=\arctan$ and the
space of approximation is ${\mathcal H}^1$.

The full procedure is defined as:
\begin{itemize}
\item In the classical sampling approach, each signal $X(t)$ is low-pass filtered at frequency $2\pi N$,
then sampled at the Nyquist frequency (the sampling pace is thus
$\frac 1 N$). Finally, it is reconstructed as $\widetilde X_l(t)$
using the cardinal series. In our examples, all signals compactly
supported (i.e., time limited) on $[0,10]$, and therefore the
summation is finite (there are $10N$ terms).
\item For our algorithm, each signal is decomposed into the system
  $\{\gamma_n\}_{0\leq n\leq N-1}$ and is then reconstructed
  as $$\widetilde X_w(t)=\sum_{n=0}^{N-1}c_n\gamma_n(t)$$
\item In each case we compute the uniform error, the $L^2$ error and
  the $H_1$ error.  All errors are relative errors:
$$E=\frac{\|X-\widetilde X\|}{\|X\|}$$
\item This process is repeated for several values of $N$, typically
  $N=1,...,78$.
\end{itemize}
This procedure allows us to study and compare the quality of
approximation of both methods in terms of different metrics. Notice
that in each case, the number of terms retained to reconstruct $X$ is
$N$.  This is the cost of compression in terms of information
quantity.

\subsection{Practical implementation}

Concerning the WKS sampling, functions must be low-pass filtered.
From a practical point of view, two cases may be distinguished:
\begin{itemize}
\item Either the filtered version of $X$ has an analytic form, as is
  the case for the Riemann function $$R_s(t)=\sum_{k\geq
    0}\frac{\sin(n^s t)}{n^s}$$ In this case, bandlimiting is
  equivalent to truncating the sum, and the samples can be expressed
  in a closed form:
$$\widetilde X_l(t_k)=\sum_{0\leq k^s\leq N}\frac{\sin(n^s t_k)}{n^s}$$
The expression of the samples is then used to compute the cardinal
series.
\item Or there is no analytic form for the low-pass filtered version
  of $X$. In this case the samples must be approximated using a
  Discrete Fourier Transform (DFT). We must thus first simulate a
  "continuous" (i.e., non-sampled) version of the signal. This is
  achieved by means of a large sampling rate.  More precisely, we
  first discretize the signal on $[0,10]$ with a regular grid of
  $10^5$ points. We then compute the DFT (via a Fast Fourier
  Transform, FFT) of this approximate "continuous" signal, low-pass
  filtered the Fourier transform and then we applied an inverse DFT.
\end{itemize}

As regards the implementation of the warping method, two types of
quantities must be computed. First the functions $\gamma_k$, for
$0\leq k\leq N-1$, and then the coefficients $c_k=\langle \widetilde
\gamma_k,f\rangle$. The latter are expressed as a duality product
between the signal $X$ to be analyzed and the dual functions
$\widetilde \gamma_k$.

We proceeded as follows: the functions $\gamma_k$ were
pre-computed under Maple in a discretized form (sampled on a
regular grid of gridspace equal to $10^{-3}$) using the analytic
form given in section 3 for $\alpha=1$. For the dual functions
$\widetilde \gamma_k$, it is a bit more complicated: they
correspond to $\alpha=-1$ and their expression is given in section
3 as a sum of a Dirac mass  and an analytic part. These two
components must be separated. We pre-computed the analytic part
under Maple (like for the $\gamma_k$) and $c_k$ was evaluated as
the sum of the inner product of $f$ with the analytic part, plus a
coefficient times $X(0)$ (corresponding to the duality product
with the Dirac mass).

It is to be noted that the error can also be evaluated in other
Sobolev metrics, say in $H_\nu$, where $\nu>0$. In this case, one has
to work in the Fourier domain and therefore use an FFT.

\subsection{Results}

We consider four examples of functions with increasing complexity.
In each case, we display four curves. The first one is a
superposition of the original function with the approximations
yielded by both the WKS and the warping methods with a single,
large, value of $N$. The three other graphs describe the
evolutions of the $L^{\infty}, L^2$ and $H^1$ errors of both
approximation methods as a function of $N$.

\subsubsection*{Cauchy Probability Distribution}
Our first test deals with a smooth function, namely the Cauchy
probability distribution $\frac 1 {1+t^2}$. Since we consider the
restriction of this function to $[0,10]$, we however introduce a
discontinuity at the endpoints, i.e., the values at 0 and 10
differ. The warping method appears extremely efficient when $N$
remains moderate (<30). Even for large values of $N$, this method
remains superior to the WKS algorithm as the errors for the latter
is larger (see figure \ref{Cauchy}). One reason for this is the
above mentioned discontinuity: As a consequence, the WKS
approximation is quite bad around the boundaries of the domain
$[0,10]$, because in the sum many terms are missing. On the
contrary, if one takes enough terms in the sum for the warping,
the approximation is excellent around 0, and quite good at 10.

\begin{figure}[htbp]
\begin{center}
\includegraphics[width=6.2cm]{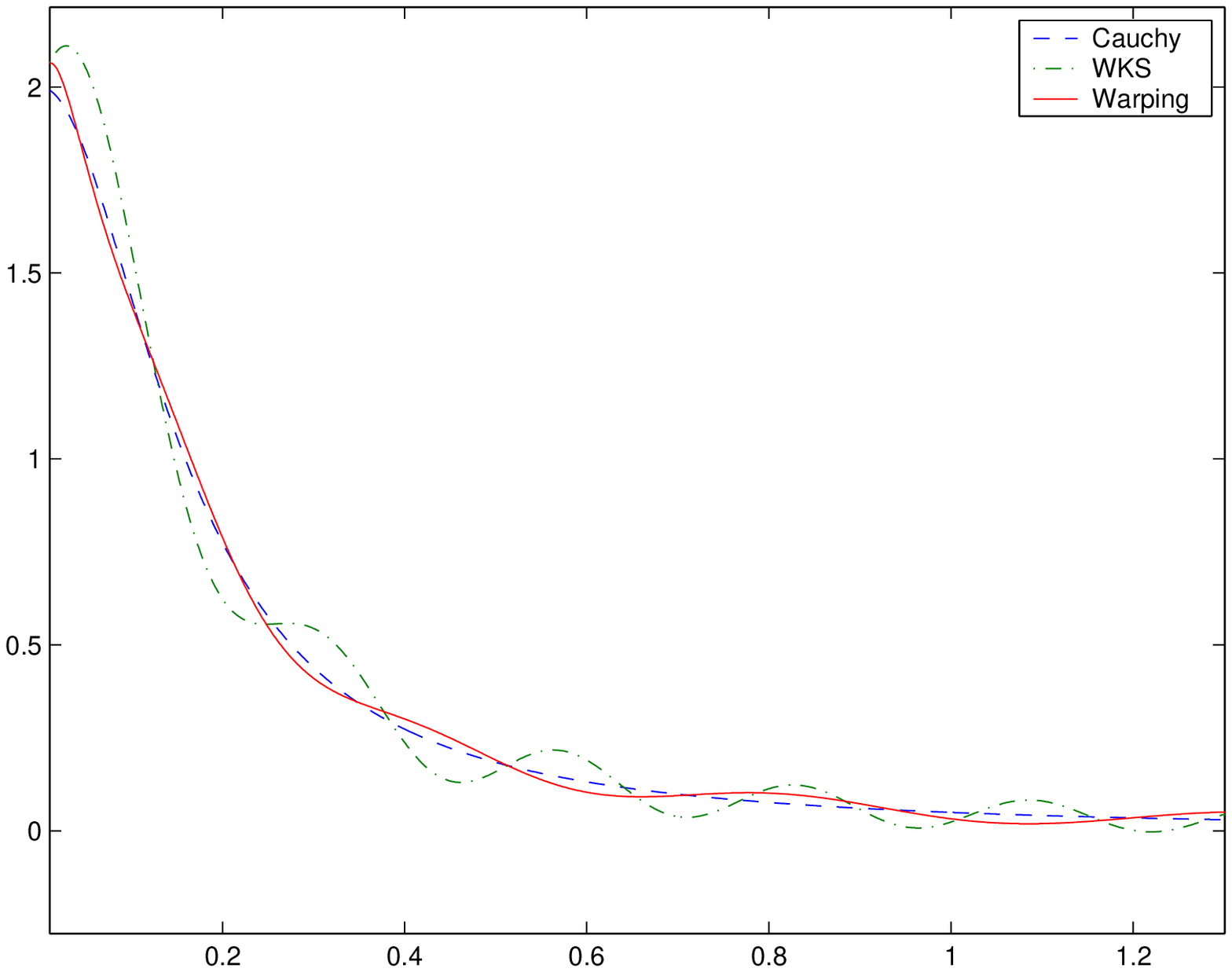}
\includegraphics[width=6.2cm]{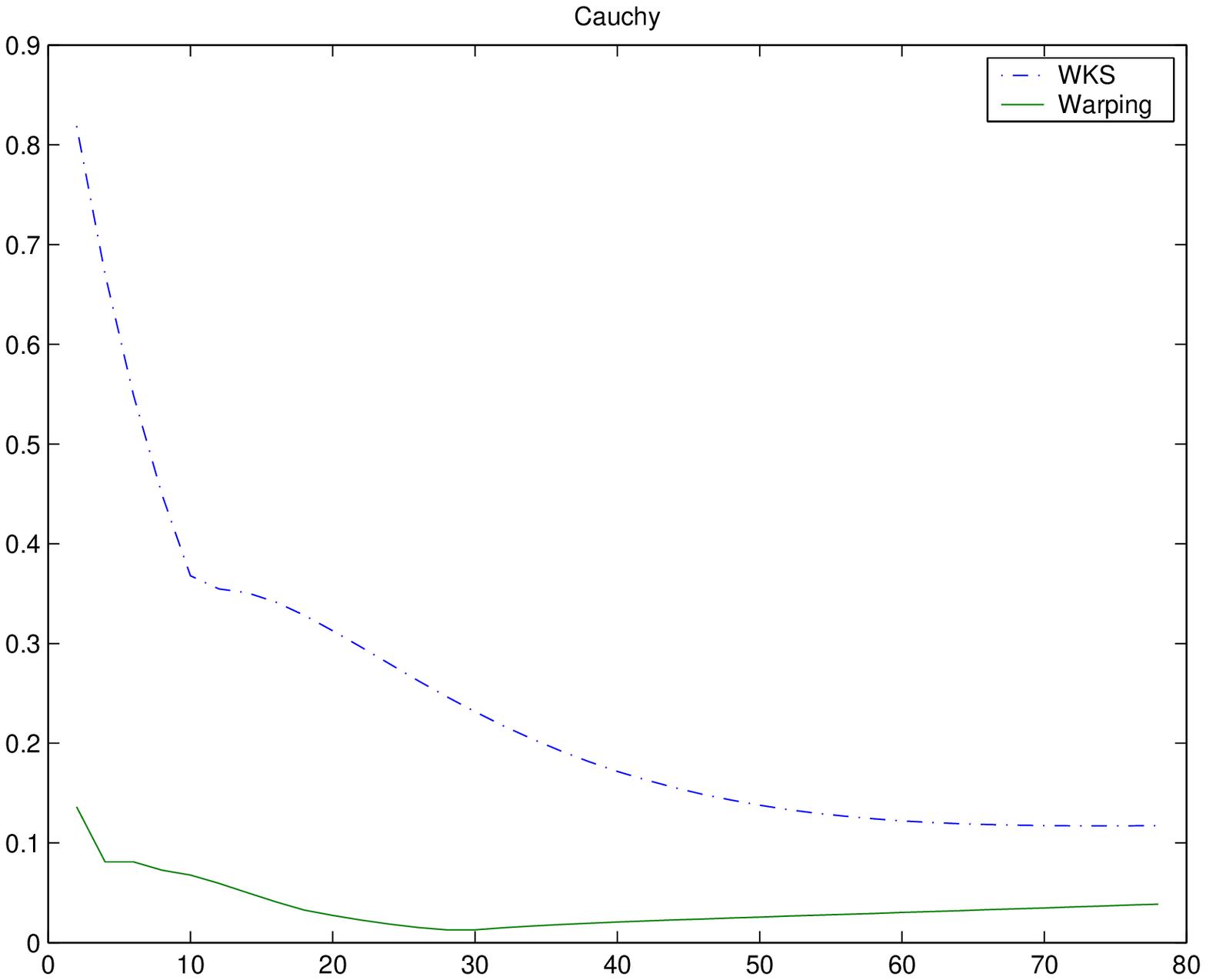}
\includegraphics[width=6.2cm]{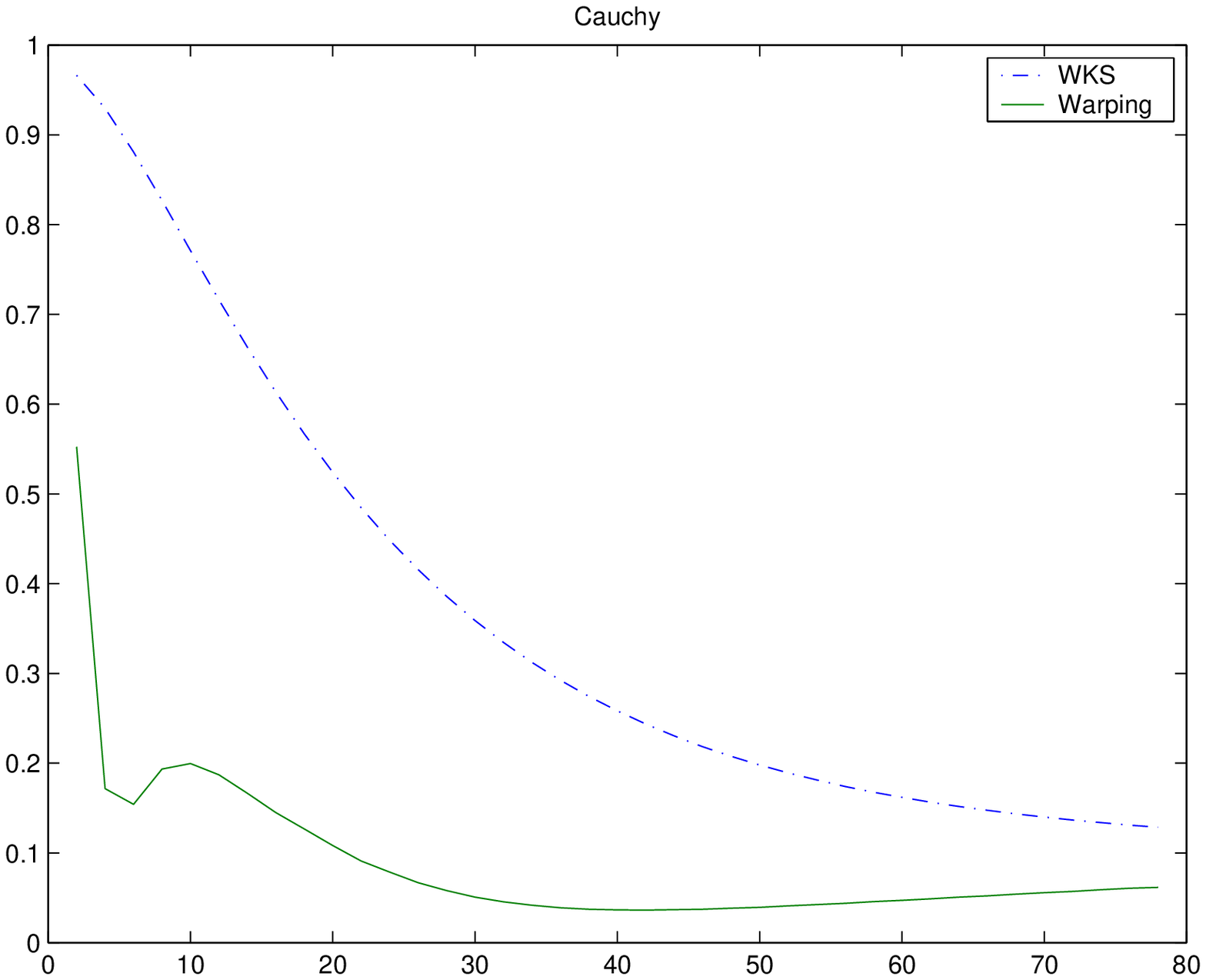}
\includegraphics[width=6.2cm]{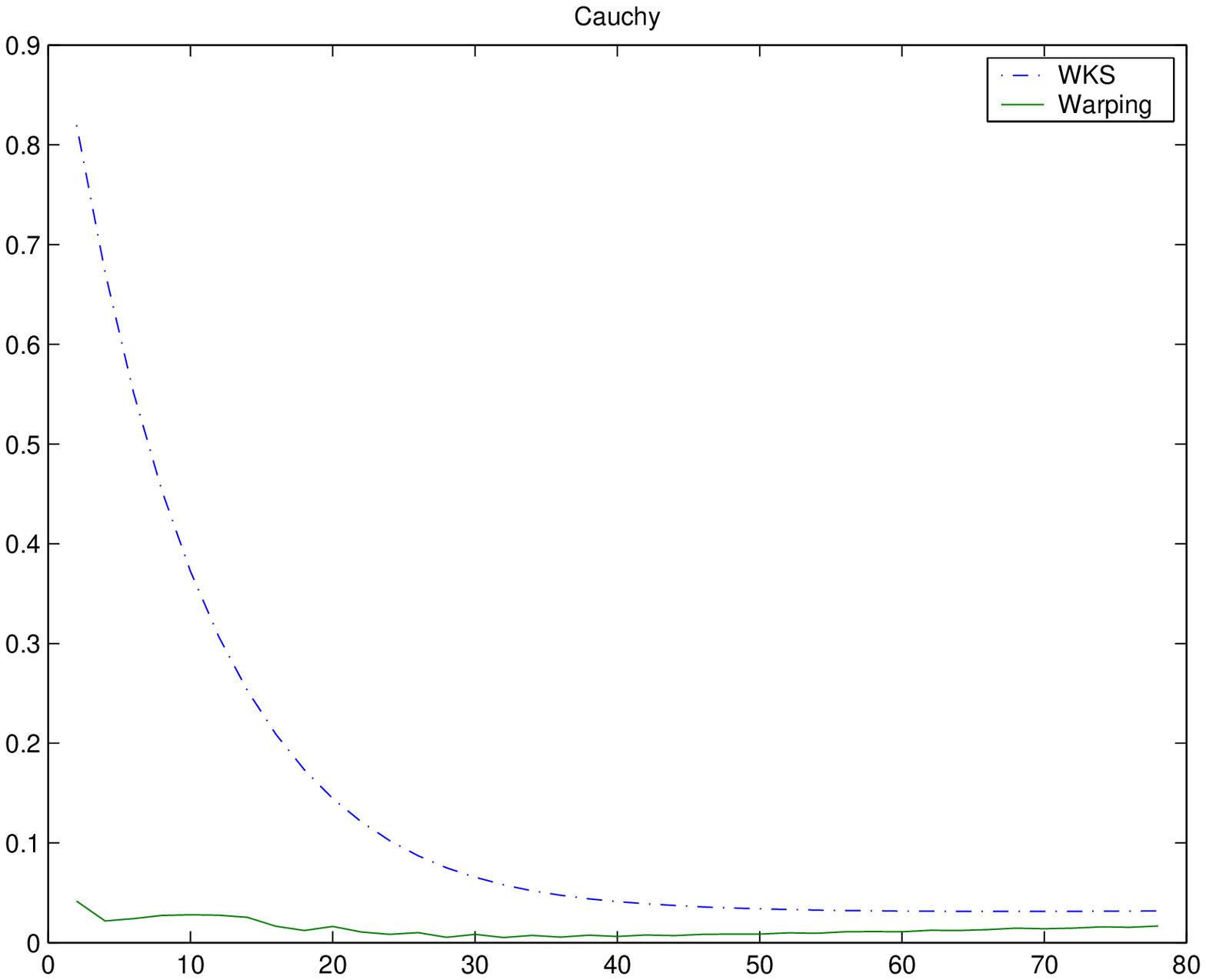}
\caption{Approximations of $\frac 1 {1+t^2}$ for $N=78$ (upper left),
  uniform error as a function of $N$ (upper right), $L^2$ error (lower
  left) and $H_1$ error (lower right).}\label{Cauchy}
\end{center}
\end{figure}

\subsubsection*{Chirp}
We now consider a function with wild oscillations, the (modified)
chirp $t\sin(\frac 1 t)e^{-t}$, restricted to $[0,10]$. The
exponential factor aims at getting an $H_1$ function that is
numerically equal to 0 at $t=10$, so as not to penalize the WKS
approximation with a boundary discontinuity. The two methods
present similar performances, although there is a slight but
noticeable gain of efficiency with the warping, as seen on figure
\ref{chirp}.

\begin{figure}[htbp]
\begin{center}
\includegraphics[width=6.2cm]{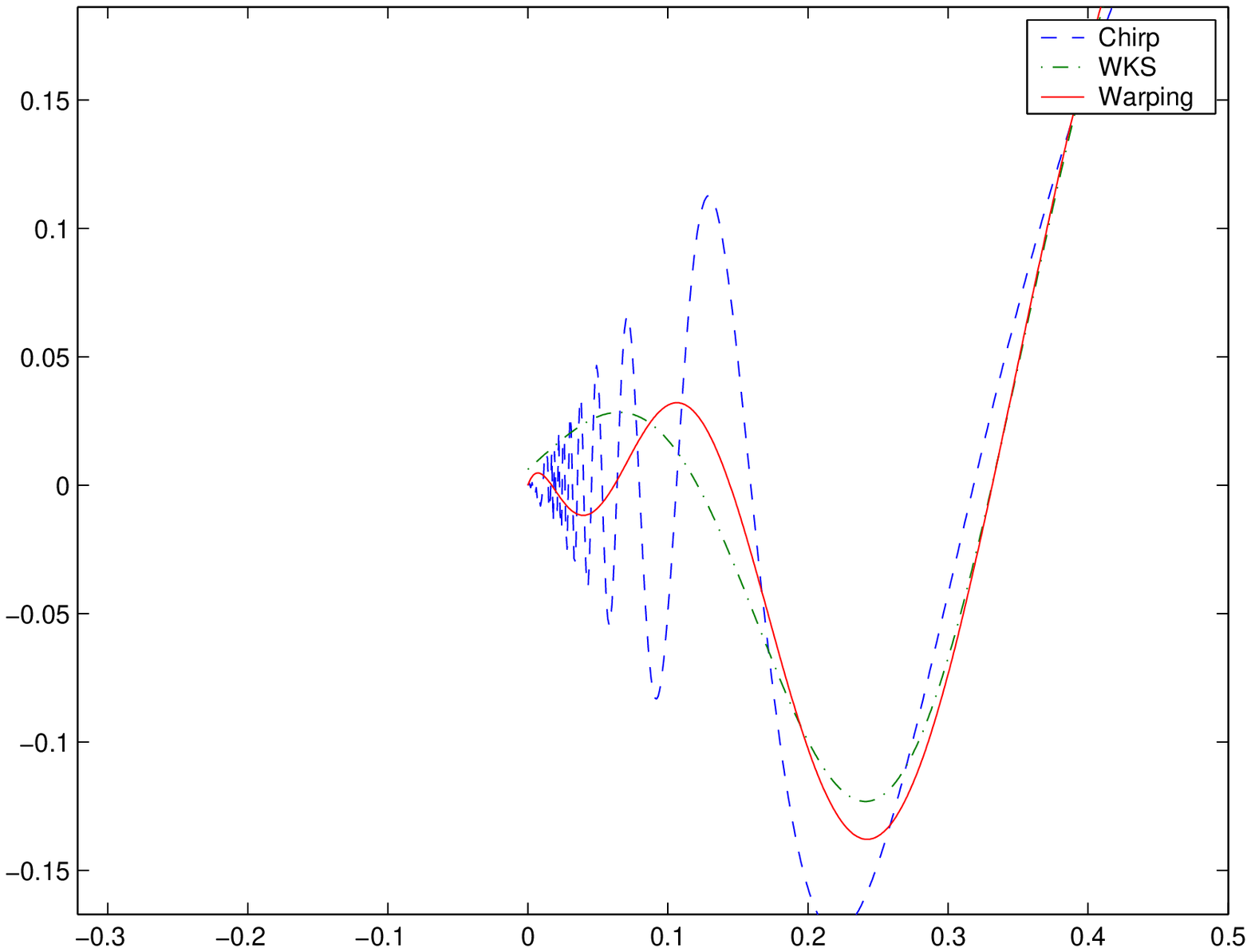}
\includegraphics[width=6.2cm]{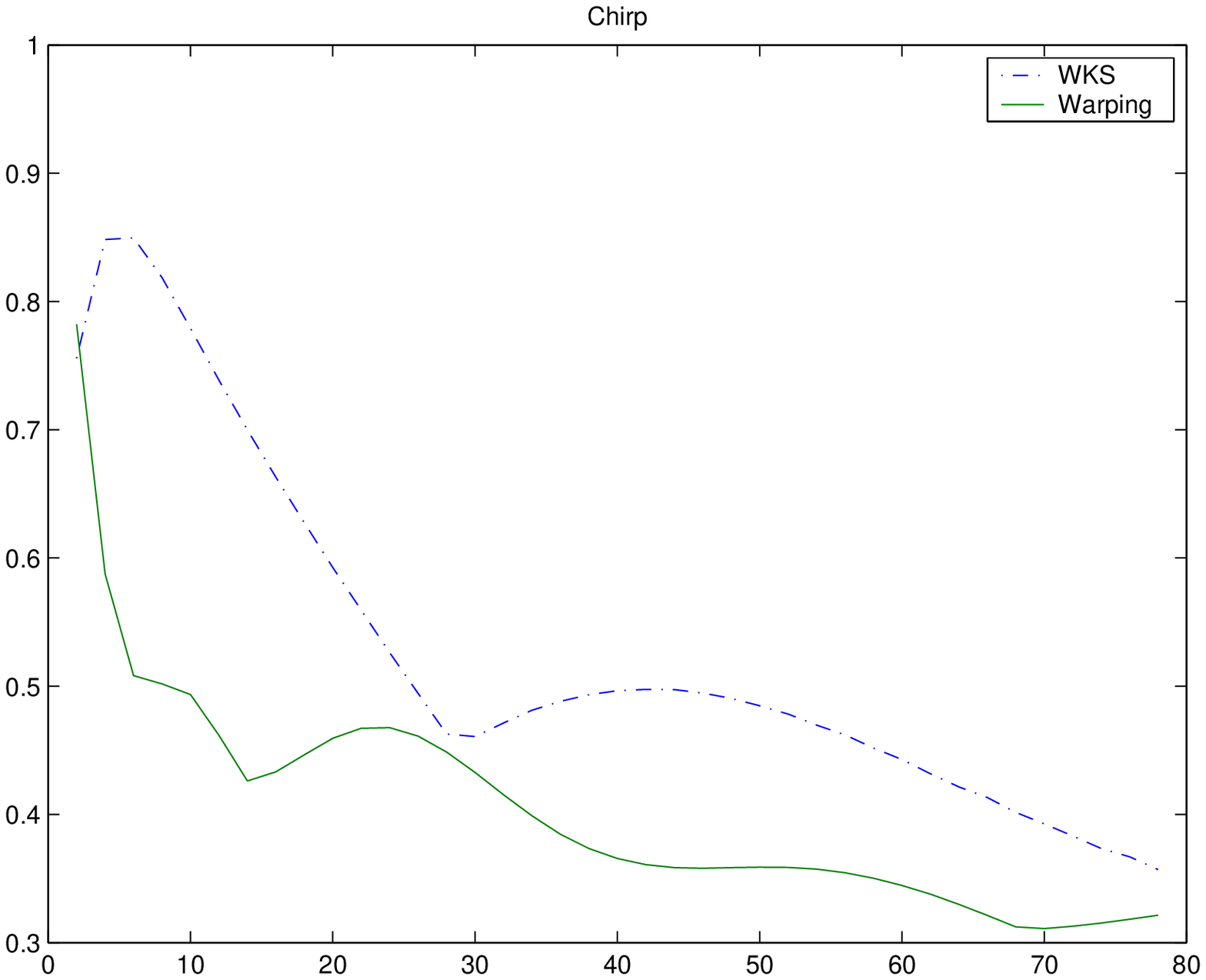}
\includegraphics[width=6.2cm]{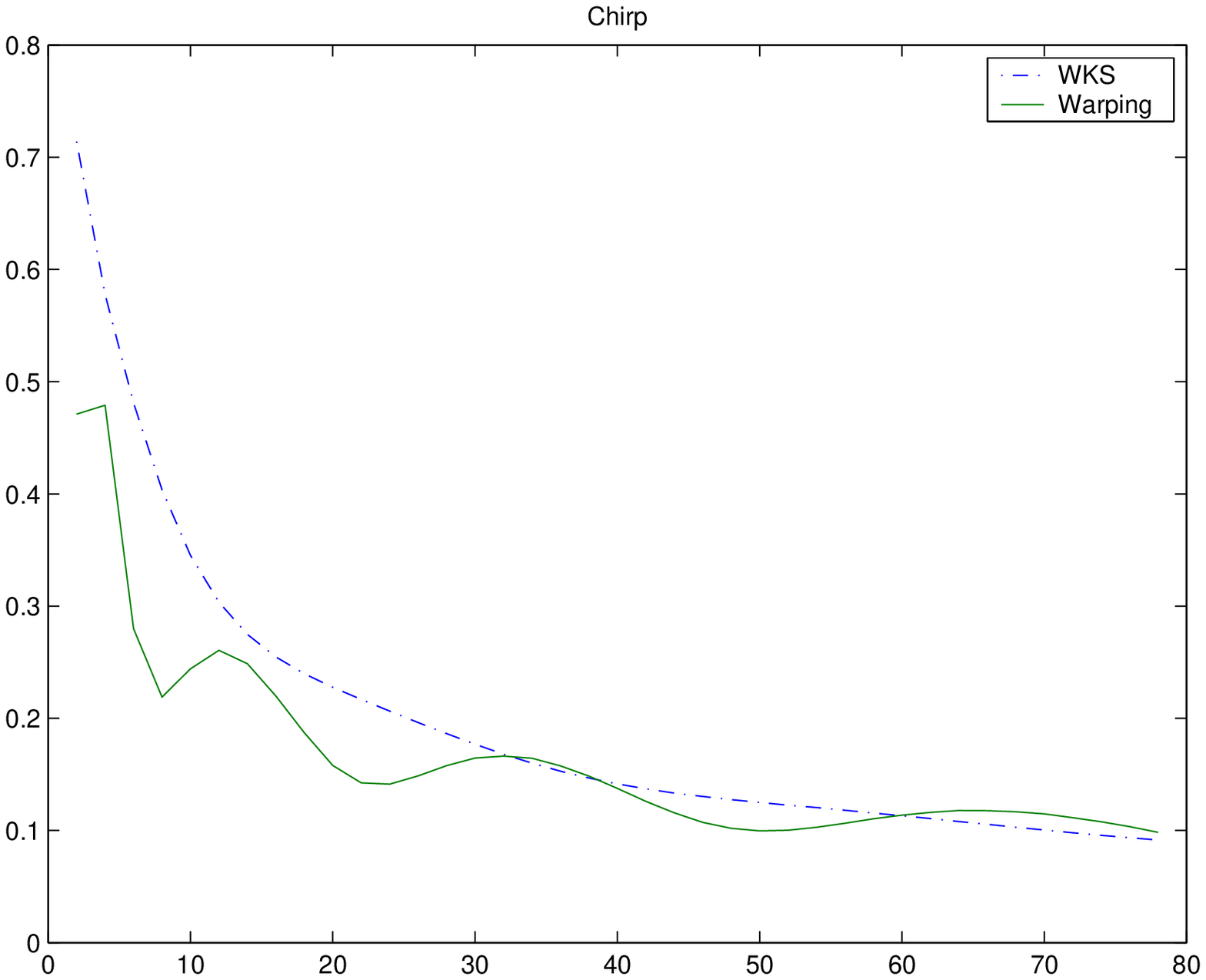}
\includegraphics[width=6.2cm]{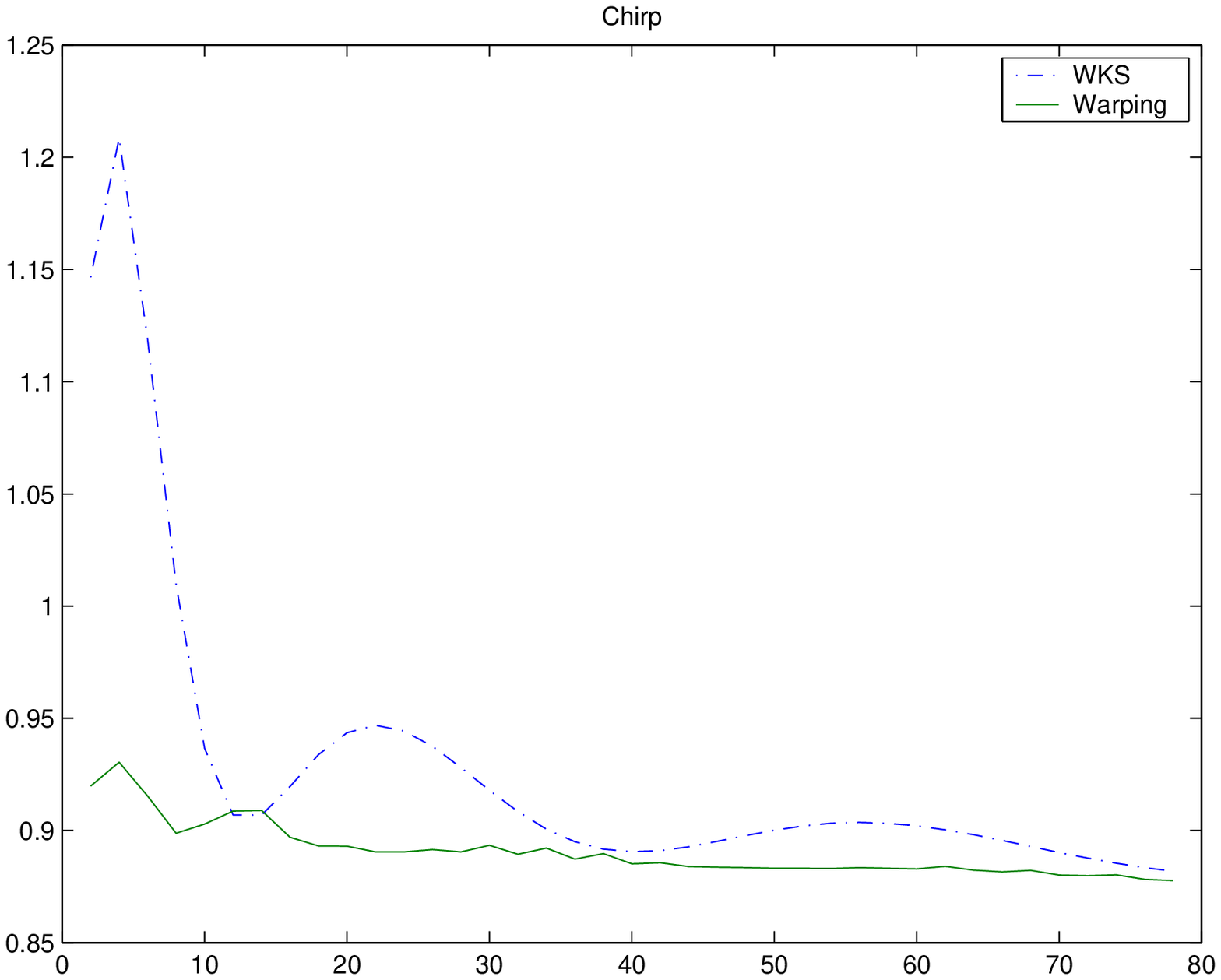}
\caption{Approximations of $t\sin(\frac 1 t)e^{-t}$ for $N=78$ (upper
  left), uniform error as a function of $N$ (upper right), $L^2$ error
  (lower left) and $H_1$ error (lower right).}\label{chirp}
\end{center}
\end{figure}

\subsubsection*{Riemann function}
We now focus on a function which is everywhere irregular and has a
multifractal structure. The Riemann function is defined as
$$R_s(t)=\sum_{n\geq 0}\frac{\sin(n^s t)}{n^s}$$
In our numerical study, we took $s=1.8$, which guarantees that
$R_s\in L^2$. As above, we analyze the restriction to $[0,10]$.
The function is "numerically equal" to zero at $10$, thus there is
no discontinuity at the endpoints.

The results are shown on figure \ref{Riemann}. Clearly the warping
method gives better results. In fact, the sine series defining
$R_s$ being lacunary, the WKS sampling method will not capture
sine waves for increasingly long ranges of values of $N$, and
consequently, for $(n-1)^s\leq N<n^s$, the WKS approximation
$\widetilde X_l$ remains the same (see the steps on figure
\ref{Riemann}). On the contrary, the warping approximation
improves steadily as $N$ grows.

\begin{figure}[htbp]
\begin{center}
\includegraphics[width=6.2cm]{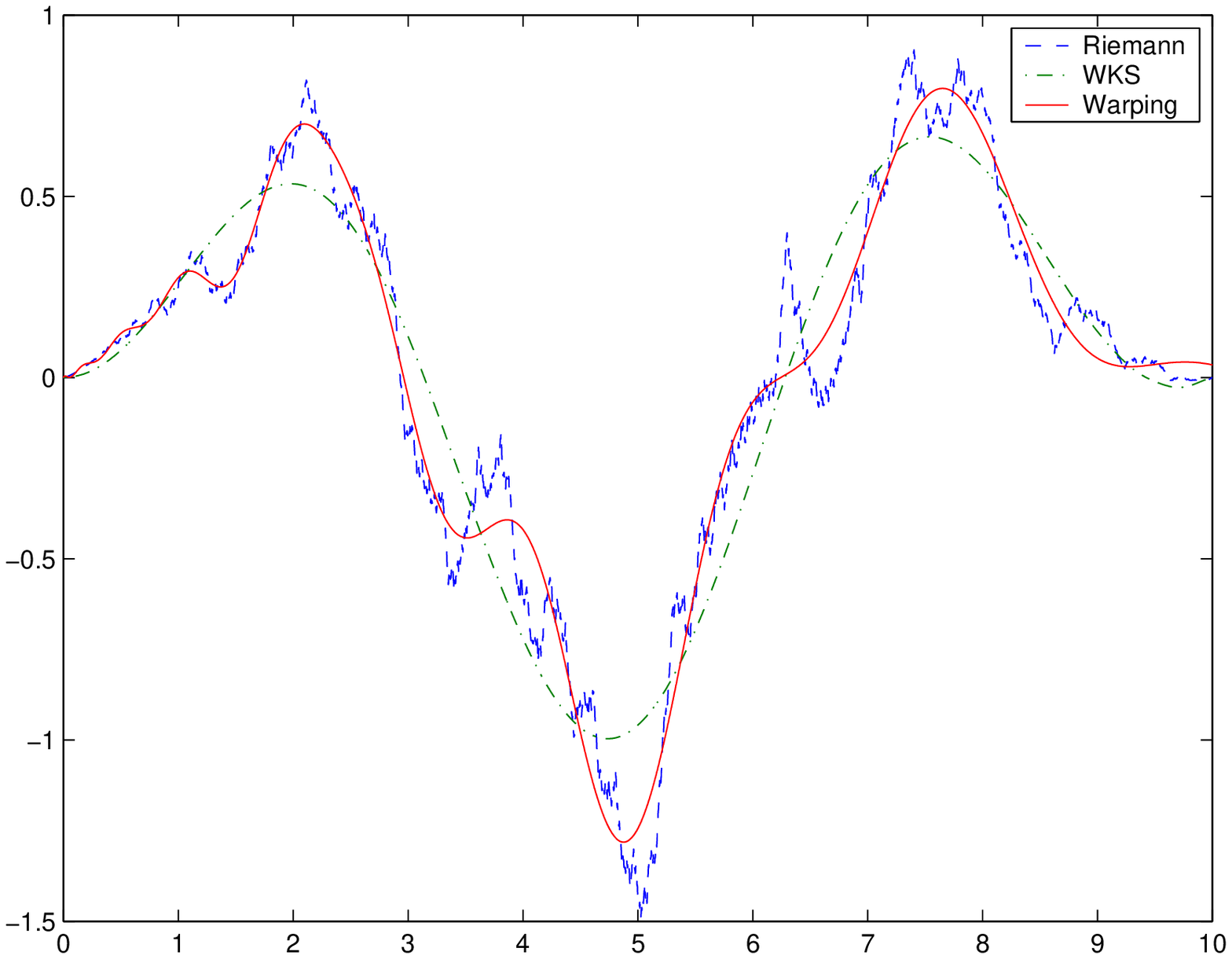}
\includegraphics[width=6.2cm]{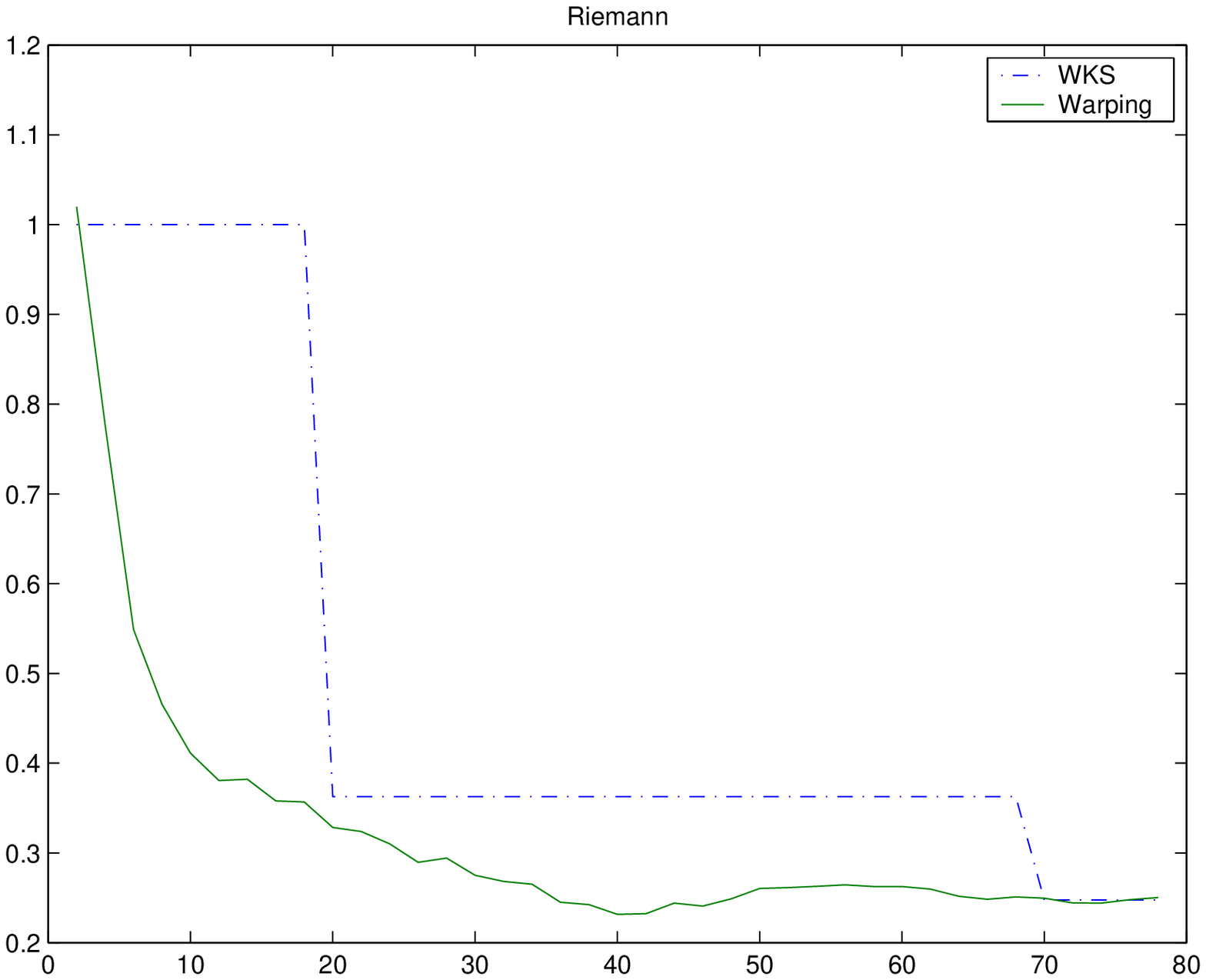}
\includegraphics[width=6.2cm]{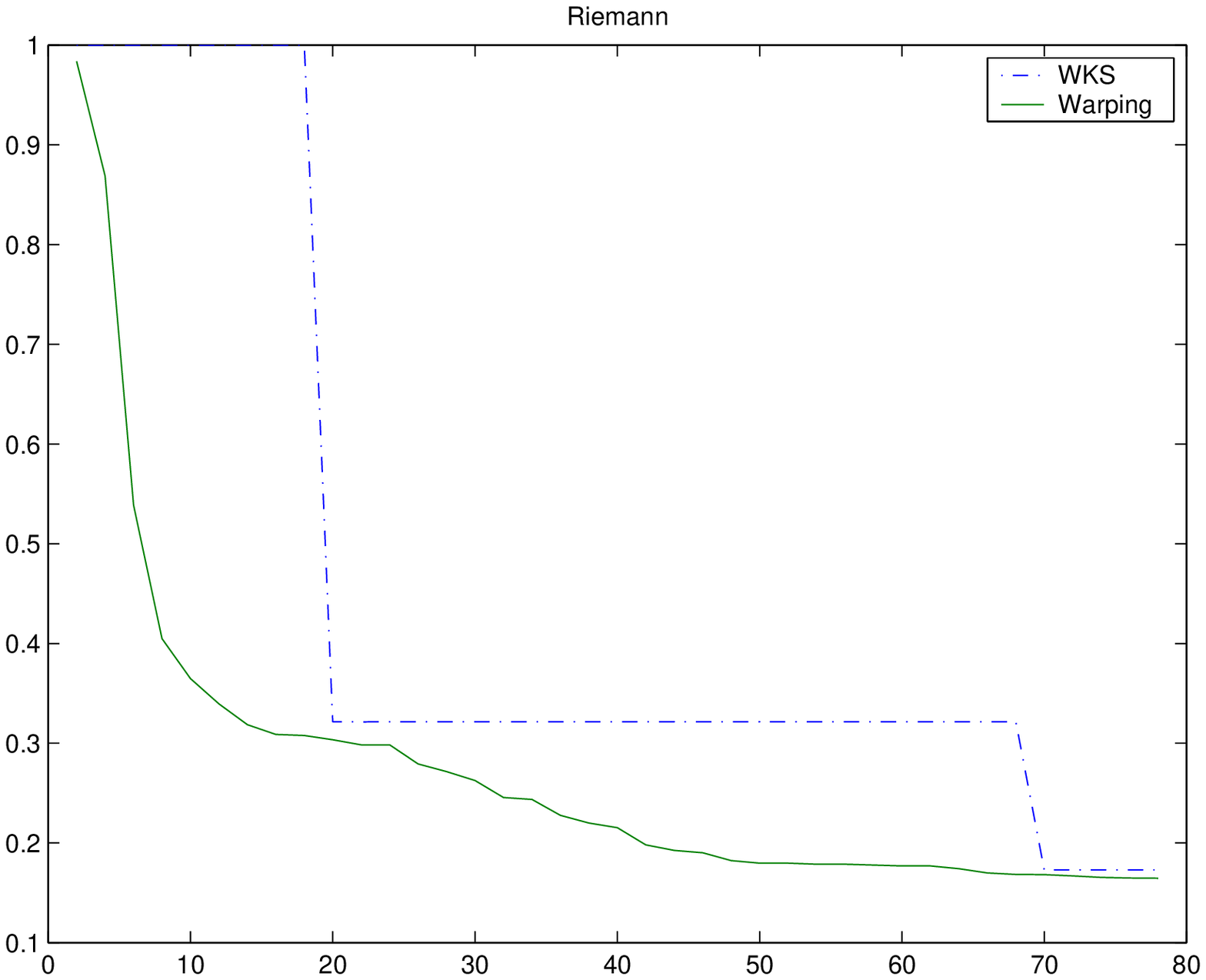}
\includegraphics[width=6.2cm]{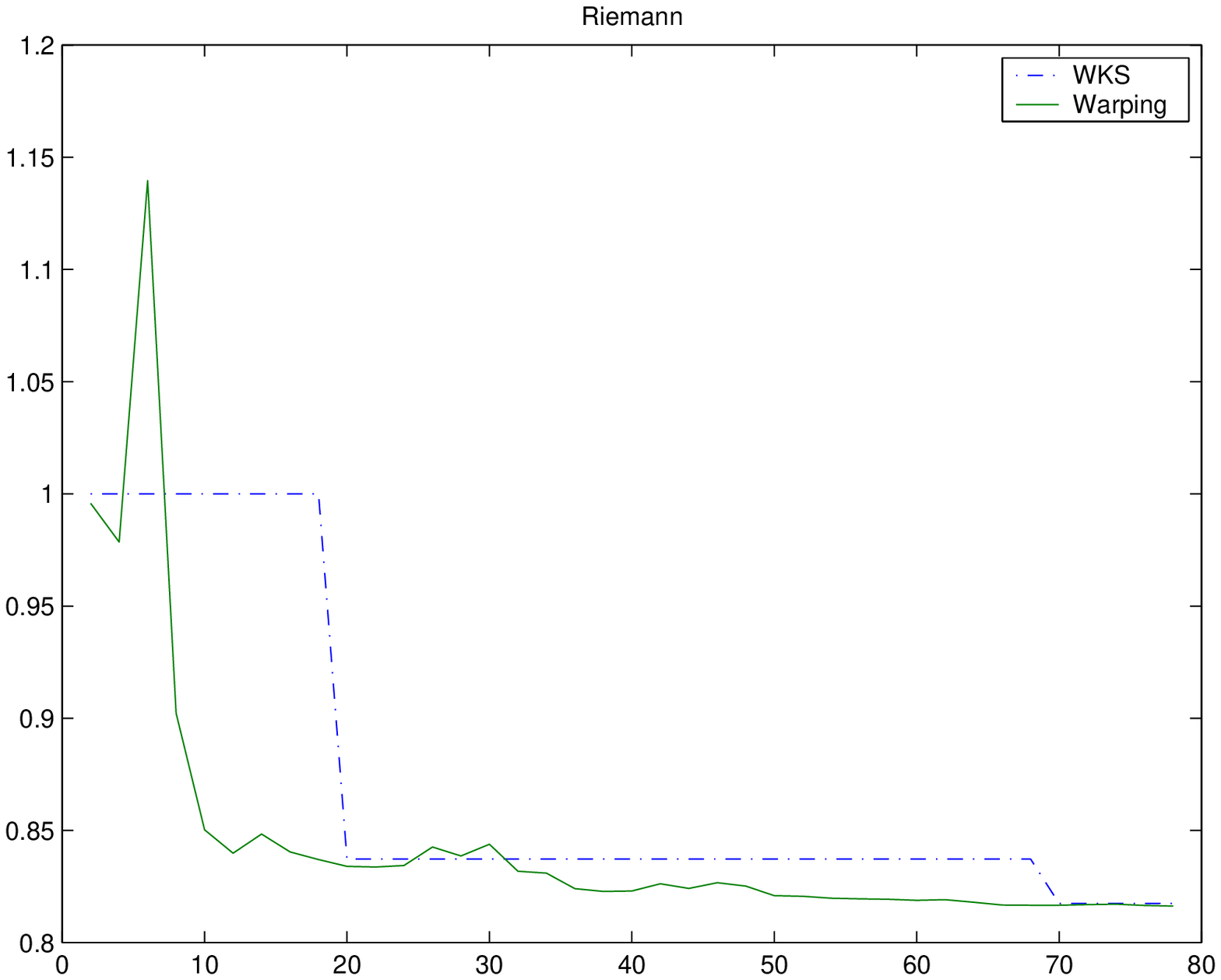}
\caption{Approximations of the Riemann function $R_s$ for $N=50$
  (upper left), uniform error as a function of $N$ (upper right),
  $L^2$ error (lower left) and $H_1$ error (lower
  right).}\label{Riemann}
\end{center}
\end{figure}

\subsubsection{Weierstrass function}
The same phenomenon occurs for the (modified)
Weierstrass function
$$W_h(t)=e^{-\frac{t^2}{\sigma^2}}\sum_{k\geq 0}\lambda^{-kh}\sin(\lambda^k
t)$$

where $h >0$, $\lambda \geq 2$ and $\sigma >0$. This function has
everywhere a H\"{o}lder exponent equal to $h$. The Gaussian factor
allows to obtain an $L^2$ function which is numerically 0 at
$t=10$, and with a Fourier transform:
$$\widehat W_h(\omega)=\sum_{k\geq 0}2^{-kh}e^{-\frac{(\omega-2\pi
    2^k)^2\sigma^2}{2}}$$ so that $W_h$ belongs to the spaces
${\mathcal W}_m^{\mu}$ for all $0 < m < h$ and $\mu \leq h -m$.

From the results in Section \ref{approx}, warping-based sampling
will yield a better approximation than WKS sampling as soon as $h
> 3/2$, with the following choice of the parameters:

$\alpha\beta=h - \frac{3}{2} - 2\varepsilon$, with $2\varepsilon <
h -3/2$, \hspace{0.4cm} $\frac{1}{2} < \beta < \frac{1}{2} +
\frac{\varepsilon}{2h - 1 - 2\varepsilon}$ \hspace{0.2cm} (when
$m=1$).

\smallskip

In order to show that the warping method may outperform WKS sampling
also for functions that do not belong to a space ${\mathcal W}_m^\mu$,
we studied numerically the case $\lambda=2, h=0.8$ and $\sigma=1$. The
results are on figure~\ref{Weierstrass}. Again, the sine series is
lacunary and the WKS error remains constant on large ranges of values
of $N$, resulting in strongly different behaviors for the WKS and
warping approximations.

\begin{figure}[htbp]
\begin{center}
\includegraphics[width=6.2cm]{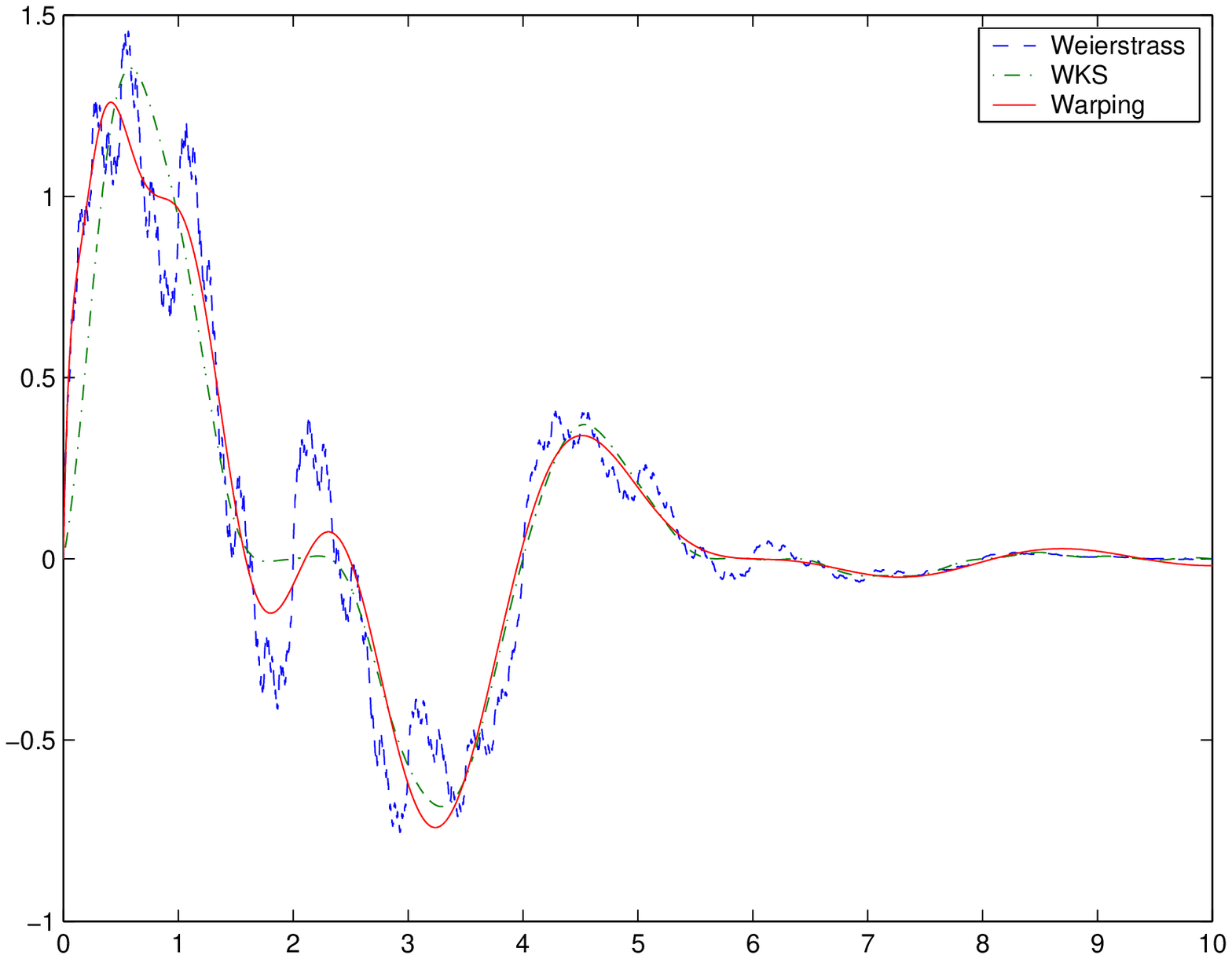}
\includegraphics[width=6.2cm]{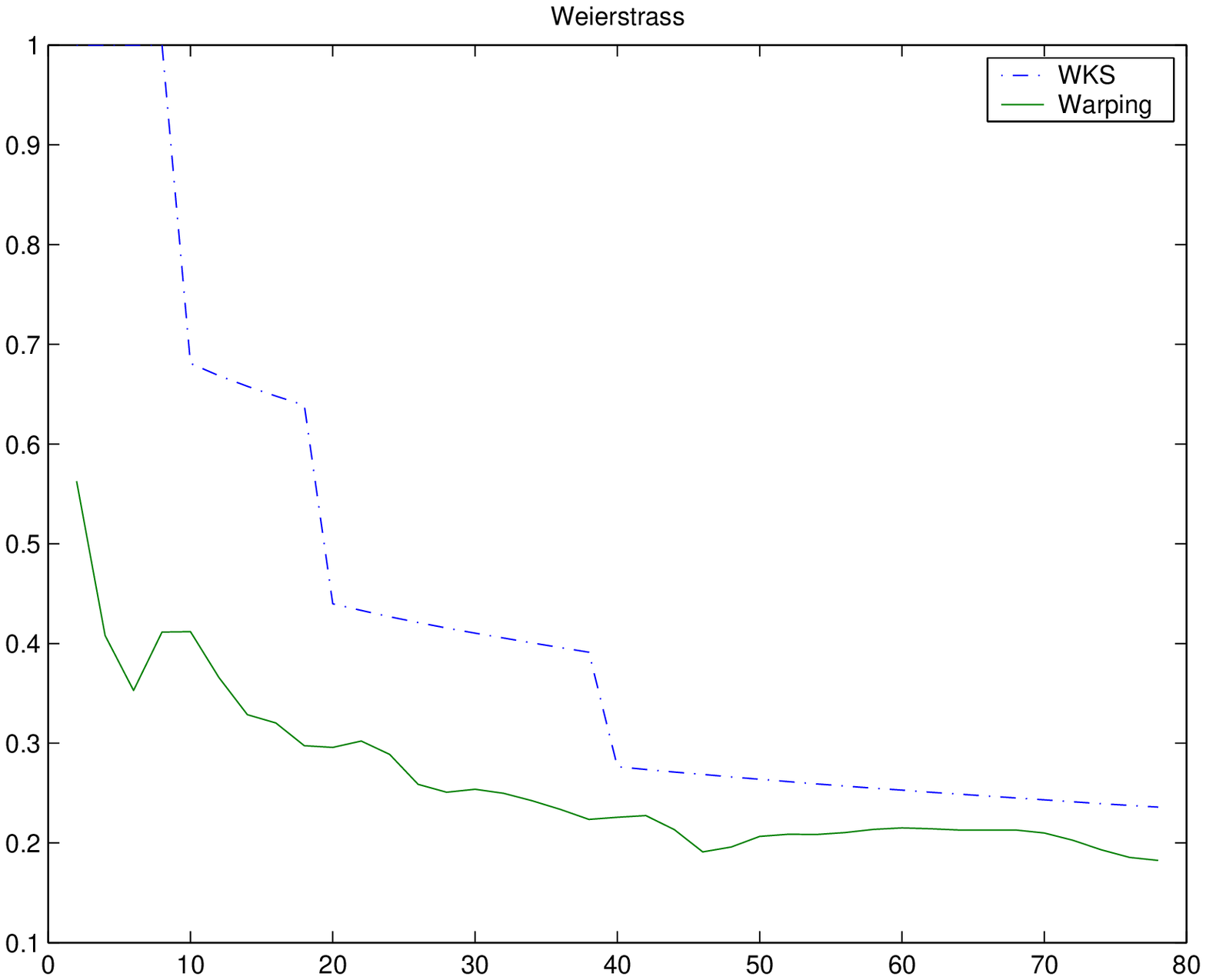}
\includegraphics[width=6.2cm]{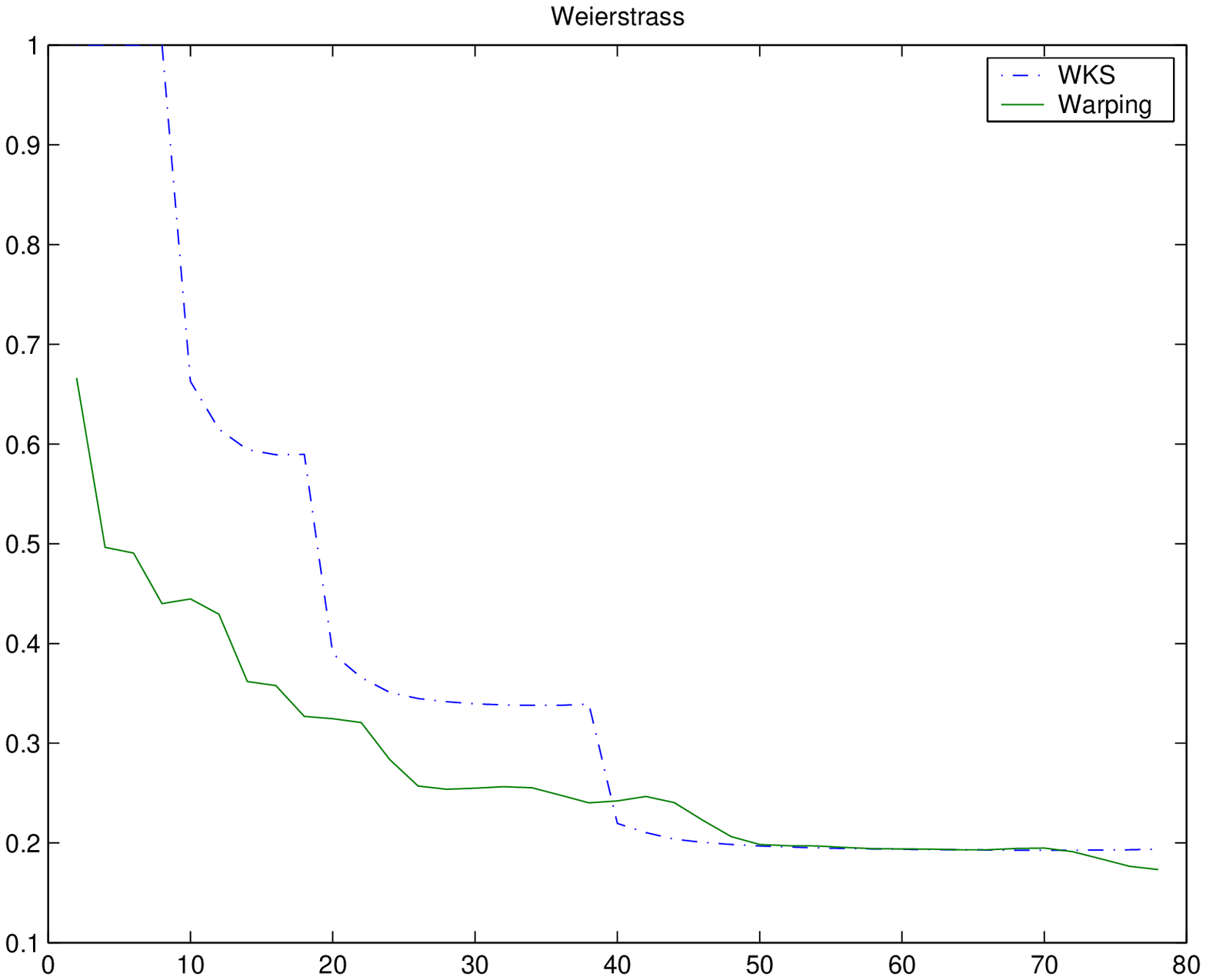}
\includegraphics[width=6.2cm]{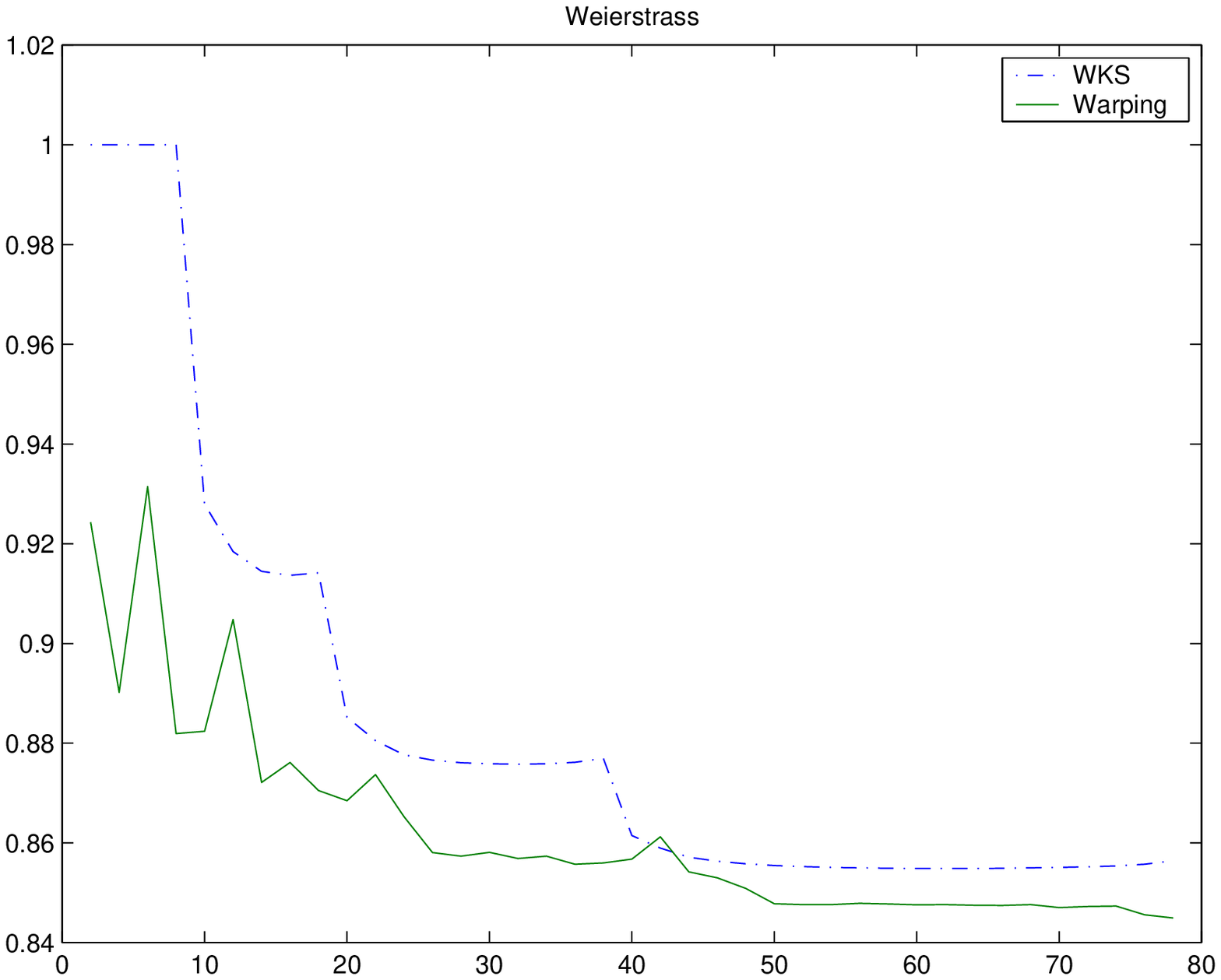}
\caption{Approximations of the Weierstrass function $W_h(t)$ for
  $N=30$ (upper left), uniform error as a function of $N$ (upper
  right), $L^2$ error (lower left) and $H_1$ error (lower
  right).}\label{Weierstrass}
\end{center}
\end{figure}

We conclude this section by mentioning how to treat the case where
$\alpha\neq 1$ or $\beta\neq 1$. This is necessary in order to process
with maximum accuracy functions in spaces ${\mathcal W}_m^\mu$ with
various values of $m$ and $\mu$.

First, if $\alpha\neq 1$ and $\beta=1$, we can use the expressions for
$\gamma_k$ given in section 3, and proceed exactly as for the case
$\alpha=1$ (i.e., $\psi=\arctan$), that is by pre-computing $\gamma_k$
and $\widetilde \gamma_k$. Of course, for the dual functions, more
Dirac masses come into play, and computing the inner product involves
the evaluation of more derivatives of $X$ at 0. If $\beta\neq 1$, the
warping function changes, and things get more complicated. There is
however a situation where analytical computations are feasible:
Indeed, when $\beta$ is an odd integer, it is possible to compute the
expression of $\gamma_k$. Therefore one can proceed as in the case
$\beta=1$. In general, there will be no closed form expression for
$\gamma_k$. These functions should then be approximated using their
expression in the Fourier domain. More precisely, one can pre-compute
the warping function $\psi$ on a grid:
$$\psi(\omega)=c\int_0^\omega \frac{dv}{(1+v^{2\beta})}$$
and use this to tabulate $$\widehat
\gamma_k(\omega)=\sqrt{\psi^\prime(\omega)}e^{2i\pi k\psi(w)}$$
Then one can use an inverse DFT to obtain an approximation of
$\gamma_k$. Likewise, the coefficients $c_k=\langle \widetilde
\gamma_k,X\rangle=\langle \widehat{\widetilde \gamma}_k,\widehat
X\rangle$ must be computed in the Fourier domain.

\end{document}